\documentclass[12pt]{amsart}
\usepackage[margin = 1in]{geometry}
\usepackage{amsmath}
\usepackage{amssymb}
\usepackage{amsthm}
\usepackage{mathtools}
\usepackage{amsfonts}
\usepackage[utf8]{inputenc}
\usepackage[dvipsnames]{xcolor}
\usepackage{graphicx}
\usepackage{hyperref}

\numberwithin{equation}{section}

\newcommand{\cone}{\mathcal{C}_2}

\newcommand{\tm}{\widetilde{m}}

\newcommand{\tP}{\widetilde{P}}

\newcommand{\tcP}{\widetilde{\cP}}

\newcommand{\tbE}{\bE_{\tm}}

\newcommand{\tH}{\widetilde{H}}

\newcommand{\Omegastar}{\Sigma_X} 

\newcommand{\deltatau}{\tau} 

\newcommand{\one}{\mathbf 1}

\let\phi\varphi

\let\hat\widehat

\newcommand{\mycenter}[2]{\left[#1\right]^{#2}}

\newcommand{\centerold}[2]{\left[#1\right]_{#2}}

\newcommand{\dd}{\; d }

\newcommand{\cP}{\mathcal{P}} 
\newcommand{\cT}{\mathcal{T}} 
\newcommand{\cB}{\mathcal{B}} 
\newcommand{\cF}{\mathcal{F}} 
\newcommand{\cC}{\mathcal{C}}

\newcommand{\cG}{\mathcal{G}}

\newcommand{\bE}{\mathbb{E}} 
\newcommand{\bN}{\mathbb{N}} 
\newcommand{\bP}{\mathbb{P}} 
\newcommand{\R}{\mathbb{R}} 
\newcommand{\ol}{\overline}

\DeclareMathOperator{\sign}{sign}

\theoremstyle{plain}

\newtheorem{teo}{Theorem}[section]
\newtheorem{prop}[teo]{Proposition}
\newtheorem{lema}[teo]{Lemma}
\newtheorem*{lema*}{Lemma}
\newtheorem{coro}[teo]{Corollary}
\newtheorem{rem}[teo]{Remark}

\newtheorem{notation}[teo]{Notation}

\begin{document}

\date{\today}


\title{Large deviations and central limit theorems for sequential and
  random systems of intermittent maps}

\author[M. Nicol]{Matthew Nicol}
\address{Matthew Nicol\\ Department of Mathematics\\
University of Houston\\
Houston\\
TX 77204\\
USA} \email{nicol@math.uh.edu}
\urladdr{http://www.math.uh.edu/~nicol/}

\author[F. Perez Pereira]{Felipe Perez Pereira}
\address{Felipe Perez Pereira\\ School of Mathematics\\
University of Bristol\\
University Walk, Bristol\\
BS8 1TW, UK} \email{fp16987@bristol.ac.uk}
\urladdr{https://felperez.github.io}

\author[A. T\"or\"ok]{Andrew T\"or\"ok}
\address{Andrew T\"or\"ok\\ Department of Mathematics\\
  University of Houston\\
  Houston\\
  TX 77204\\
  USA and 
{Institute of Mathematics of the Romanian Academy, Bucharest, Romania.}}
\email{torok@math.uh.edu}
\urladdr{http://www.math.uh.edu/~torok/}

\thanks{MN was supported in part by NSF Grant DMS 1600780. FPP thanks the
  University of Houston for hospitality while this work was completed. FPP
  was partially supported by the Becas Chile scholarship scheme from
  CONICYT and thanks the University of Houston for hospitality. AT was supported in part by NSF Grant DMS 1816315.}

\begin{abstract}
 We obtain   large and moderate deviations estimates for both sequential and random compositions of intermittent maps. 
 We also address the question of whether or not centering 
is necessary for the quenched central limit theorems (CLT) obtained by Nicol, T\"or\"ok and Vaienti for random dynamical systems comprised of  intermittent maps. Using recent work of Abdelkader and Aimino, Hella and Stenlund  we  extend  the results of Nicol, T\"or\"ok and Vaienti  on
 quenched central limit theorems (CLT) for centered observables over random compositions of intermittent maps: first by enlarging the parameter range over which the quenched CLT holds; and second by showing  that the variance in the quenched CLT is almost surely constant (and the same as the variance of the  annealed CLT) and that centering is needed to obtain this quenched CLT.
\end{abstract}

\maketitle

\section{Introduction}

The theory of limit laws and rates of decay of correlations for uniformly hyperbolic and some non-uniformly hyperbolic sequential  and random dynamical systems has recently seen major progress. Results in this area include: in \cite{Conze_Raugi} strong laws of large numbers and 
centered central limit theorems for sequential expanding maps; in \cite{Aimino:2015ab}, polynomial  decay of correlations for sequential intermittent systems; in \cite{Nicol:2018aa}, sequential and quenched (self-centering) central limit theorems for intermittent systems; in \cite{Aimino:2015aa}, annealed versions of a central limit theorem, large deviations principle, local limit theorem and almost sure invariance principle are proven for random expanding dynamical systems, as well as quenched versions of a central limit theorem, dynamical Borel-Cantelli lemmas, Erd\H{o}s-R\'enyi laws and concentration inequalities; in \cite{Abdelkader:2016aa}, necessary and sufficient conditions are given for a central limit theorem without random centering for uniformly expanding maps; and  in \cite{Bahsoun:2016ab} mixing rates and 
central limit theorems are given for random intermittent maps using a Tower construction. Recently the preprint
\cite{Bahsoun_Bose_Ruziboev} considered quenched decay of correlation for slowly mixing systems and 
the preprint \cite{Aimino_Freitas} used martingale techniques to obtain  large deviations for systems with
stretched exponential decay rates. 

In this article we obtain  large deviations estimates for both sequential and random compositions of intermittent maps. We also address the question of whether or not centering 
is necessary for the quenched central limit theorems (CLT) obtained in \cite{Nicol:2018aa} for random dynamical systems comprised of  intermittent maps. More precisely, we consider in the first instance a fixed  deterministically chosen sequence of maps 
$\hdots T_{\alpha_n},\dots, T_{\alpha_1}$
in the sequential case, or a randomly drawn sequence $\hdots T_{\omega_n},\dots, T_{\omega_1}$ with respect to a Bernoulli measure $\nu$ on $\Sigma := 
\{T_1,\hdots, T_k \}^\mathbb{N}$, where each of the maps $T_j$ is a Liverani-Saussol-Vaienti~\cite{LivSauVai} intermittent map of form
\begin{align*}
    T_{\alpha_j}(x) = \left.
  \begin{cases}
    x + 2^{\alpha_j} x^{1+\alpha_j}, & 0\leq x\leq 1/2, \\
    2x - 1, & 1/2 \leq x \leq 1 
  \end{cases}
  \right.,
\end{align*}
for numbers $0 < \alpha_j\leq \alpha <1$. We consider the asymptotic
behavior of the centered (that is, after substracting their expectation)
sums
\begin{align*}
    S_n := \sum_{k=1}^n \varphi \circ (T_{\alpha_k}\circ\hdots\circ T_{\alpha_1})
\end{align*}
for sufficiently regular observables $\varphi$.

Denote by $m$  Lebesgue measure on $X:=[0,1]$, and by $m(\varphi)$ the
integral of $\varphi$ with respect to $m$. We will also consider the
measure $\tm$ given by $d\tm (x) = x^{-\alpha} d m$, where $0<\alpha_j \le \alpha <1$. The motivation for
introduction of this measure is that in the case of a stationary system, if
$\alpha_k = \alpha$ for each $k$, then a natural and  convenient measure to use is the 
invariant measure $\mu_{\alpha}$ for  $T_\alpha$, which behaves near 0 as $x^{-\alpha}$. In the stationary case large deviation estimates are
given with respect to $\mu_{\alpha}$ and $m$ in ~\cite{Melbourne:2008aa} for $\alpha <\frac{1}{2}$ and for all $0\le \alpha <1$ in ~\cite{Mel2009}.

In the sequential case of a fixed realization we are interested in the
large deviations of the self-centered sums:
\begin{align*}
    m\left\{ x : \dfrac{1}{n}\left|\sum_{k=1}^n \varphi \circ (T_{\alpha_k}\circ\hdots\circ T_{\alpha_1}) - \sum_{k=1}^n m(\varphi \circ T_{\alpha_k}\circ\hdots\circ T_{\alpha_1}) \right|  > \epsilon \right\}
\end{align*}
for $\epsilon>0$.  We also obtain large deviations with respect to $\tm$,
which are in a sense sharper. In the sequential  case centering is clearly necessary.

In the annealed case we consider the random dynamical system (RDS)
$F\colon \Sigma\times [0,1]\to \Sigma\times [0,1]$ given by
$F(\omega,x)=(\tau\omega,T_{\alpha_1}x)$ for
$\omega=(\alpha_1, \alpha_2, \dots)\in \Sigma$, where $\tau$ is the
left-shift operator on $\Sigma$. For $\nu$ a Bernoulli measure on $\Sigma$,
we suppose $\mu$ is a stationary measure for the stochastic process on
$[0,1]$, that is, a measure such that $\nu\otimes \mu$ is $F$ invariant.
This assumption is valid in the setting we consider. If $\varphi$ is an
observable such that $\mu (\varphi)=0$, we estimate
\begin{align*}
  \nu\otimes\mu\left\{ (\omega,x) : \dfrac{1}{n}\left|\sum_{k=1}^n
  \varphi \circ (T_{\alpha_k}\circ\hdots\circ T_{\alpha_1}) \right|  >
  \epsilon \right\}.
\end{align*}

In the quenched case, once again assuming $\mu (\varphi)=0$, we give  bounds for
\begin{align*}
    m\left\{ x :  \dfrac{1}{n}\left|\sum_{k=1}^n \varphi \circ (T_{\alpha_k}\circ\hdots\circ T_{\alpha_1}) \right|  > \epsilon \right\}
\end{align*}
for  $\nu$-almost every realization $\omega\in\Sigma$.

Since the maps we are considering are not uniformly hyperbolic, spectral methods used to obtain limits laws are not immediately available.
Our techniques to establish large and moderate deviations estimates  are based on those developed for stationary systems, in 
particular the  martingale methods  of \cite{Melbourne:2008aa,Mel2009}.

Using recent work of \cite{Abdelkader:2016aa} and \cite{2020HellaStenlund}
we extend the results of \cite{Nicol:2018aa} on quenched central limit
theorems (CLT) for centered observables over random compositions of
intermittent maps in two ways, first by enlarging the parameter range over
which the quenched CLT holds and second by showing as a consequence of
results in~\cite{2020HellaStenlund} that the variance in the quenched CLT
is almost surely constant and equal to the variance of the annealed
CLT.

We also study the necessity of centering to achieve a quenched CLT using ideas of \cite{Abdelkader:2016aa} and \cite{Aimino:2015aa}. The work of ~\cite{Aimino:2015aa} together with our observations show that centering is necessary `generically' (in a sense made precise later) to obtain the quenched CLT in  fairly general hyperbolic situations.

\subsection*{Improvements of earlier results} With this paper we improve
some results of \cite{Nicol:2018aa}:

\begin{itemize}

\item we show that the sequential CLT in \cite[Theorem 3.1]{Nicol:2018aa},
  \cite{2019HellaLeppanen}, holds for the sharp $\alpha< 1/2$ (from
  $\alpha <1/9$) if the variance grows at the rate specified.

\item we show that the CLT holds not only with respect to Lebesgue measure
  $m$ but also for $d \tm=x^{-\alpha}d m$, which scales at the origin as
  the invariant measure of $T_\alpha$.
  
\item in the case of quenched CLT's of \cite[Theorem 3.1]{Nicol:2018aa},
  using results of Hella and Stenlund~\cite{2020HellaStenlund} we show that
  the variance $\sigma_{\omega}^2$ is almost-surely the same for any
  sequence of maps and equal to the annealed variance $\sigma^2$.

\end{itemize}

\begin{rem}
  After this work was finished we learned about a preprint by Korepanov
  and Lepp\"{a}nen~\cite{KorepanovLeppanen:2020}, in which interesting related
  results are obtained.
\end{rem}

\section{Notation and assumptions}

Throughout this article, $m$ denotes the Lebesgue measure on $X:=[0,1]$ and
$\cB$ the Borel $\sigma$-algebra on $[0,1]$. We consider the family of
intermittent maps given by
\begin{align}\label{intermittent}
    T_\alpha(x) = \left.
  \begin{cases}
    x + 2^\alpha x^{1+\alpha}, & 0\leq x\leq 1/2, \\
    2x - 1, & 1/2 \leq x \leq 1 
  \end{cases}
  \right.,
\end{align}
for $\alpha\in(0,1)$.

For $\beta_k\in(0,1)$
denote by $P_{\beta_k} = P_k \colon L^1(m) \to L^1(m)$ the transfer
operator (or Ruelle-Perron-Frobenius operator)  with respect to $m$ associated to the map
$T_{\beta_k} = T_k$, defined as the ``pre-dual'' of the Koopman operator
$f\mapsto f\circ T_k$, acting on $L^\infty(m)$. The
duality relation is given by
\begin{align*}
    \int_X P_k f \ g \ dm = \int_X f \ g\circ T_k \ dm
\end{align*}
for all $f\in L^1(m)$ and $g\in L^\infty(m)$~\cite[Proposition
4.2.6]{Boyarsky_Gora}. For a fixed sequence $\{\beta_k \}$ such that
$0<\beta_k\leq\alpha$ for all $k$, define
\begin{align*}
  \cT^{\infty} := & \hdots,T_{\beta_n}, \hdots , T_{\beta_1}\\
  \cT^{n}_m := &T_{\beta_n} \circ\hdots\circ T_{\beta_m}, \qquad  \cT^{n} :=\cT^{n}_1 \\ 
  \cP^{n}_m := &P_{\beta_n}\circ\hdots\circ P_{\beta_m}, \qquad  \cP^{n} :=\cP^{n}_1 
\end{align*}
We will often write, for ease of exposition when there is no ambiguity,
$T_{\beta_n} \circ\hdots\circ T_{\beta_m}$ as $T_n\circ\hdots\circ T_m$ and
$P_{\beta_n}\circ\hdots\circ P_{\beta_m} $ as $P_n\circ\hdots\circ P_m$.

Since $L^1(m)$ is invariant under the action of the transfer operators, the
duality relation extends to compositions
\begin{align*}
  \int_X \cP^{n}_k f \ g \ dm = \int_X f \ g\circ \cT^{n}_k \ dm.
\end{align*}
We will write $\bE_m [\varphi |\mathcal{F}]$ for the conditional
expectation of $\varphi$ on a sub-$\sigma$-algebra $\mathcal{F}$ with
respect to the measure $m$. To simplify notation we might write $\bE$ for
$\bE_m$.

\begin{rem}
    In \cite{Conze_Raugi,Nicol:2018aa} it is shown that
    \begin{equation}\label{eq:conditional-expectation}
      \bE_m[\varphi\circ\cT^\ell | \cT^{-k}\cB] = \dfrac{
        P_{k}\circ\hdots\circ P_{\ell +1}
        (\varphi\cdot\cP^\ell(\mathbf{1}))}{\cP^k(\mathbf{1})}\circ\cT^k
    \end{equation}
    for $0\leq \ell\leq k$. 
\end{rem}

One of the main tools to study sequential and random systems of intermittent maps is the use of cones (see \cite{LivSauVai}, \cite{Aimino:2015ab}, \cite{Nicol:2018aa} ). Define the cone ${\cone}$ by
\begin{align*}
    {\cone} := \{ f\in \mathcal{C}^0((0,1])\cap L^1(m) \ | \ f\geq 0, f \text{ non-increasing }, X^{\alpha+1}f \text{ increasing }, f(x)\leq ax^{-\alpha}m(f) \},
\end{align*}
where $X(x) = x$ is the identity function and $m(f)$ is the integral of $f$ with respect to $m$. In \cite{Aimino:2015ab} it is proven that for a fixed value of $\alpha\in(0,1)$, provided that the constant  $a$ is big enough, the cone ${\cone}$ is invariant under the action of all transfer operators $P_\beta$ with $0<\beta\leq\alpha$. 

 \begin{notation}
   In general we will denote the transfer operator with respect to a
   non-singular\footnote{The measure $\mu$ is non-singular for the
     transformation $T$ if $\mu(A)>0 \implies \mu(T(A))>0$.} measure $\mu$
   (not necessarily Lebesgue measure) by $P_\mu$. Similarly, the
   (conditional) expectation will be denoted by $\bE_\mu$.

  Denote the centering with respect to $\mu$ of a function
  $\phi \in L^1(X, \mu)$ by
  \begin{equation}\label{eq:centering}
    \mycenter{\phi}{\mu}:=\phi - \frac{1}{\mu(X)}\int_X \phi \; d \mu
  \end{equation}
  In particular, for $g(x):=x^{-\alpha}$, denote the measure $g m$ by
  $\tm$, the corresponding transfer operator by $\tP := P_{g m}$, and the
  (conditional) expectation by $\tbE:=\bE_{g m}$.

\end{notation}

 \subsection*{Random dynamical systems.}
 
 \let\newalpha\beta
 
 Now we introduce a randomized choice of maps: consider a finite family of
 intermittent maps of the form (\ref{intermittent}), indexed by a set
 $\Omega = \{\newalpha_1,\hdots,\newalpha_m \}\subset (0,\alpha)$. Given a
 probability distribution $\bP =(p_1,\hdots,p_m)$ on $\Omega$, define a
 Bernoulli measure $\bP^{\otimes\bN}$ on $\Sigma:=\Omega^\bN$ by
 $\bP^{\otimes\bN}\{\omega : \omega_{j_1} =
 \newalpha_{j_1},\hdots,\omega_{j_k} = \newalpha_{j_k} \} = \prod_{i=1}^k
 p_{{j_i}}$ for every finite cylinder and extend to the sigma-algebra
 generated by the cylinders of $\Sigma$ by Kolmogorov's extension theorem.
 This measure is invariant and ergodic with respect to the shift operator
 $\tau$ on $\Sigma$, $\tau\colon\Sigma\to\Sigma$ acting on sequences by
 $(\tau(\omega))_k = \omega_{k+1}$. We will denote $\bP^{\otimes\bN}$ by
 $\nu$ from now on.

 For $\omega =(\omega_1, \omega_2, \dots) \in \Sigma$ define
 $\cT^n_\omega := T_{(\tau^n\omega)_1}\circ\hdots\circ T_{\omega_1} =
 T_{\omega_n}\circ\hdots\circ T_{\omega_1} $. The random dynamical system
 is defined as
  \begin{align*} 
   F :  \Sigma \times X  \to \Sigma \times X 
   \\
   (\omega, x)  \mapsto\left(\tau \omega, T_{\omega_{1}} x\right).
 \end{align*}
 The iterates of $F$ are given by
 $F^n(\omega,x) = (\tau^n(\omega),\cT_{\omega}^n(x))$.

 We will also use $\Omega$-indexed subscripts for random transfer operators
 associated to the maps $T_{\omega_i}$, so that
 $P_{\omega_i}:=P_{T_{\omega_i}}$. We will also abuse notation and write
 $P_{\omega}$ for $P_{\omega_1}$ if
 $\omega=(\omega_1,\omega_2,\ldots, \omega_n, \ldots)$.

A probability measure $\mu$ on $X$ is said to be stationary with respect to
the RDS $F$ if
\begin{align*}
  \mu(A)=\int_{\Sigma} \mu\left(T_{\omega_1}^{-1}(A)\right) d \nu(\omega) =
  \sum_{\beta\in \Omega} p_{\beta} \mu\left(T_{\beta}^{-1}(A)\right)
\end{align*}
for every measurable set $A$, where $p_\beta$ is the $\bP$-probability of
the symbol $\beta$. This is equivalent to the measure $\nu\otimes\mu$ being
invariant under the transformation $F \colon  \Sigma \times X  \to \Sigma \times X$.

See Remark~\ref{stationary-measure} about the existence and ergodicity of
such a stationary measure in our setting.

The annealed transfer operator $P \colon L^{1}(m) \rightarrow L^{1}(m)$ is defined by averaging over all the transformations:
\begin{align*}
  P=\sum_{\beta \in \Omega} p_{\beta} P_{\beta} = \int_{\Sigma}
  P_\omega \ d\nu(\omega).
\end{align*}
This operator is ``pre-dual'' to the annealed Koopman operator
$U \colon L^{\infty}(m) \rightarrow L^{\infty}(m)$ defined by
\begin{align*}
  (U\varphi) (x):=\sum_{\beta \in \Omega} p_{\beta} \varphi (T_{\beta} x)=
  \int_{\Sigma}  \varphi (T_{\omega} x)  d\nu(\omega) =  \int_{\Sigma}
  F(\tilde\varphi) (\omega, x)  d\nu(\omega)
\end{align*}
where $\tilde{\phi}(\omega, x) := \phi(x)$. The annealed operators satisfy
the duality relationship
\begin{align*}
    \int_X (U\varphi) \cdot \psi \ dm = \int_X \varphi \cdot P\psi \ dm
\end{align*}
for all observables $\varphi\in L^\infty(m)$ and $\psi\in L^1(m)$.

\section{Background results and the Martingale approximation}

In this section we describe the main technique used to prove some of the
limit law results: the martingale approximation, introduced by Gordin
\cite{Gordin1969aa}. Since there is no common invariant measure for the set
of maps $\{ T_k\}$, for a given $\mathcal{C}^1$ observable $\varphi$ we
center along the orbit by
\begin{align*}
  \centerold{\varphi}{k} (\omega,x) &:= \varphi (x)-  \int_X
                                      \varphi\circ\cT_\omega^k \ dm,
\end{align*}
with $\cT_\omega^k= Id$ for $k=0$.

This implies that $\bE_m (\centerold{\varphi}{k}\circ\cT^k) = 0$ and consequently the
centered Birkhoff sums
\[
  \hat{S}_n:= \sum_{k=1}^n \centerold{\varphi}{k} \circ \cT^k,
\]
have zero mean with respect to $m$. Following \cite{Nicol:2018aa}, define
\begin{equation}\label{eq:defn-H_n}
  H_1 := 0 \text{ and } H_n\circ \cT^n := \bE_m(\hat{S}_{n-1}|\mathcal{B}_n)
  \text { for } n\ge 2
\end{equation}
and the (reverse) martingale sequence $\{ M_n \}$ by 
\begin{align*}
  M_ 0 := 0 \text { and }  \hat{S}_n = M_n + H_{n+1}\circ\cT^{n+1},
\end{align*}
where the filtration here is $\mathcal{B}_n = \cT^{-n}\mathcal{B}$. Define
$\psi_n \in L^1(m)$ by setting
\begin{align*}
  \psi_n = \centerold{\varphi}{n} + H_n - H_{n+1}\circ T_{n+1},     
\end{align*}
then $M_n-M_{n-1}=\psi_n \circ \cT^n$ and we have that
$\bE(M_n | \mathcal{B}_{n+1}) = 0$. Thus $\{ \psi_n \circ \cT^n \}$ is a
reverse martingale difference scheme. An explicit expression for $H_n$ is
given by
\begin{align}\label{eq:H_n}
  H_n = \dfrac{1}{\cP^n\mathbf{1}}\left[
  P_n(\centerold{\varphi}{n-1}\cP_{n-1}\mathbf{1}) +
  P_nP_{n-1}(\centerold{\varphi}{n-2}\cP_{n-2}\mathbf{1}) + \hdots + P_nP_{n-1}\hdots
  P_1(\centerold{\varphi}{0}\cP_{0}\mathbf{1}) \right].
\end{align}

\begin{rem}
  The formulas derived so far with $m$ being the Lebesgue measure actually
  hold for any measure $\mu$ that is non-singular for the transformations
  $T_\beta$ considered. The conditional expectations $\bE_{\mu}$ will be
  with respect to $\mu$ and the transfer operator $P_{\mu}$ will be with
  respect to the measure space $(X,\mu)$. In particular the centering will
  have the form
  \begin{align*}
    \centerold{\varphi}{k} (\omega,x) &:= \varphi (x)- \frac{1}{\mu(X)} \int_X
                                        \varphi\circ\cT_\omega^k  \ d \mu,
\end{align*}
but all other equations are the same, with the notational changes just
described.
\end{rem}

We collect and extend some results from \cite{Nicol:2018aa} concerning the
properties of $H_n$, as well as the non-stationary decay of correlations
for the sequential system.

We state first a few formulas for changing from a measure $m$ to the
measure $g(x) \dd m(x)$ with $g\in L^1(m)$; for simplicity, we denote this
new measure as $g m$ when there is no possibility of confusion.

\begin{lema}[Change of measure]\label{change-of-measure}
  We state this result only for the situation we need, but it holds also  for any measure $\mu$ non-singular with
  respect to $T$ in place of $m$ the Lebesgue measure, and instead of $g(x)=x^{-\alpha}$ 
 for any  $g \in L^1(\mu)$, $g > 0$.

  Note that $L^1(g m) = g^{-1} L^1(m)$, so all formulas below make sense
  for $\phi$ in the appropriate $L^1$-space. 
  
  We have:
  \begin{align}
    \nonumber
    m(\phi) & = m(P_m \phi)
    \\
    \label{eq:change-of-measure}
    P_{g m}(\phi) & = g^{-1} P_{m}( g \phi)
    \\
    \nonumber
    g \mycenter{\phi}{g m} & = \mycenter{g\phi}{m} - \frac{m(g
                             \phi)}{m(g)} \mycenter{g}{m}
    \\
    \nonumber
    \bE_{g m} (\phi| \cB)& = \bE_{m} (g \phi|
                                   \cB)/\bE_{m} (g | \cB)
  \end{align}
  Therefore
  \begin{equation}\label{eq:P-changed-measure}
    (\cP_{g m})_{\ell}^{k}(\mycenter{\phi}{g m})=g^{-1}(\cP_{m})_{\ell}^{k}
    \left(\mycenter{g\phi}{m}- \frac{m(g \phi)}{m(g)}
      \mycenter{g}{m}\right)
  \end{equation}
\end{lema}

\begin{proof}
  The first two properties are standard and follow from the definition of the
  transfer operator. The third is a direct computation using the
  notation~\eqref{eq:centering}.

  For the fourth, $\bE_{g m} (\phi| \cB)$ is the function $\Phi$ that is
  $\cB$-measurable and
  $\int \Phi \psi \dd (g m) = \int \phi \psi \dd (g m)$ for each
  $\psi \in L^\infty(\cB)$. Expanding the LHS,
  \[
    \int \Phi \psi \dd (g m) = \int \Phi \psi g \dd m = \int \Phi \psi
    \bE_{m}(g | \cB) \dd m
  \]
  whereas the RHS becomes
  \[
    \int \phi \psi \dd (g m) = \int \phi \psi g \dd m = \int \bE_{m}(g \phi
    | \cB) \psi \dd m
  \]
  Thus $\Phi \bE_{m}(g | \cB) = \bE_{m}(g \phi | \cB)$, as claimed.
\end{proof}

\begin{prop}[\cite{Nicol:2018aa}]\label{nicoldecay2}

  If $\phi, \psi$ are both in the cone $\cone$ and have the same mean,
  $\int_X \phi d m = \int_X \psi d m$, then by~\cite[Theorem
  1.2]{Nicol:2018aa}
  \[
    \left\|\cP^{n} (\phi)- \cP^{n} (\psi)\right\|_{L^{1}(m)} \leq
    C_{\alpha} (\|\phi\|_{L^{1}(m)}+\|\psi\|_{L^{1}(m)})
    n^{-\frac{1}{\alpha}+1}(\log n)^{\frac{1}{\alpha}}
  \]

  Moreover~\cite[Remark 2.5 and Corollary 2.6]{Nicol:2018aa}, for
  $\phi\in \mathcal{C}^1$, $h\in \cone$ and any sequence of maps
  $\cT^\infty$:
  \begin{align*}
    \left\|\cP^{n} (\mycenter{h \phi}{m})\right\|_{L^{1}(m)} \leq
    C_{\alpha}
    \mathcal{F}\left(\|\varphi\|_{\mathcal{C}^{1}}+m(h)\right) 
    n^{-\frac{1}{\alpha}+1}(\log n)^{\frac{1}{\alpha}}
  \end{align*}
  where $C_{\alpha}$ depends only on the map $T_\alpha$, and
  $\mathcal{F}\colon\mathbb{R}\to\mathbb{R}$ is an affine function.
\end{prop}

The decay result of Proposition~\ref{nicoldecay2} for products of elements
in the cone with $C^1$ observables (see also~\cite[Theorem
4.1]{LivSauVai}), follows from Lemma~\ref{C^1-into-cone}, which was stated
in \cite[proof of Theorem 4.1]{LivSauVai}. The proof of
Lemma~\ref{C^1-into-cone} is given in the Appendix; a different -- less
transparent -- proof is given in \cite[Lemma 2.4]{Nicol:2018aa}.

\begin{lema}\label{C^1-into-cone}
  Suppose $\varphi\in \mathcal{C}^1$ and $h\in {\cone}$. Then there exist
  constants $\lambda, A, B \in \R$ such that $(\varphi+A + \lambda x )h+B$
  and $(A+\lambda x)h+B$ are both in ${\cone}$ and hence if
  $\int \varphi h dm=0$ then
  $\|\cP^j (\varphi h)\|_{L^1(m)} \le C \rho(j) \|\varphi h\|_{L^1(m)}$
  where $\rho(j)$ is the $L^1(m)$-decay for centered functions from the
  cone $\cone$.
\end{lema}

Note that in our setting $\rho(j)=j^{-\frac{1}{\alpha}+1}(\log j)^{\frac{1}{\alpha}}$.

A consequence of Proposition~\ref{nicoldecay2} is the non-stationary decay
of correlations (\cite[Page 1130]{Nicol:2018aa})
\begin{align*}
  \left|\int_X \phi \cdot \psi \circ T_{\omega_n}\circ  \ldots \circ
  T_{\omega_1} dm -m(\phi)\cdot m(\psi \circ T_{\omega_n}\circ  \ldots \circ
  T_{\omega_1})\right|
  \\
  \le \|\psi\|_{\infty} \|\cP_{\omega}^n  (\phi)-\cP_{\omega}^n
  (\mathbf{1}\int_X \phi dm)\|_{L^1(m)}
\end{align*}

We derive next decay estimates with respect to the measure $\tm$, which are
better in $L^p$, $p>1$, than those for $m$.
\begin{prop}
  For $\phi:[0,1] \to \R$ bounded, $h\in \cone$ and $1\le p \le \infty$:
  \begin{align}
     \label{eq:L-infty-g}
     & \|\tcP^{n}\left(\phi\right)\|_{L^\infty(\tm)} \le m(g)
       \|\phi\|_{L^{\infty}(\tm)}
     \end{align}
     For $\phi\in C^1([0,1])$, $h\in \cone$
     \begin{align}
     \label{eq:L-1-g}
     & \|\tcP^{n}\left(\mycenter{(g^{-1} h) \phi}{\tm}\right)\|_{L^1(\tm)} 
       \leq C_\alpha
       \mathcal{F}
       \left(\|\phi\|_{\mathcal{C}^{1}}+ m(h)\right) n^{-\frac{1}{\alpha}+1}(\log
       n)^{\frac{1}{\alpha}}
     \end{align}
     and therefore, if $1\le p \le \infty$,
     \begin{align}
       \label{eq:L-p-g}
       & \|\tcP^{n}\left(\mycenter{(g^{-1} h)
         \phi}{\tm}\right)\|_{L^p(\tm)} \le C_\alpha^{\frac{1}{p}}
         \left(m(g) \|\phi\|_{L^{\infty}(\tm)}\right)^{1-\frac{1}{p}}
         \mathcal{F}^{\frac{1}{p}}\left(\|\phi\|_{\mathcal{C}^{1}} +
         m(h)\right)
         n^{\frac{1}{p}\left(-\frac{1}{\alpha}+1\right)}(\log
         n)^{\frac{1}{p\alpha}}
     \end{align}
     where $C_{\alpha}$ depends only on $T_\alpha$ and $\mathcal{F}$ is an
     affine function.

  Note that the $L^1$ and $L^p$ bounds are relevant
  only for $\phi\in C^1$.
\end{prop}

\begin{proof}
  The $L^1$ and $L^\infty$ bounds give \eqref{eq:L-p-g}, since
  \begin{equation}\label{eq:L-p-interpolation}
    \|f\|_{L^p} \le \|f\|_{L^\infty}^{1-\frac{1}{p}}
    \|f\|_{L^1}^{\frac{1}{p}}
  \end{equation}
  because
  \[
    \int |f|^p \le \int \|f\|_{L^\infty}^{p-1} |f| = \|f\|_{L^\infty}^{p-1}
    \|f\|_{L^1}.
  \]

  To prove the $L^\infty$ estimate \eqref{eq:L-infty-g} note that by the
  invariance of the cone $\cone$, $\cP^{n}\left(g\right) \in \cone$, so
  $\cP^{n}\left(g\right)\le x^{-\alpha} m(\cP^{n}\left(g\right)) =
  x^{-\alpha} m(g)$. That is, using \eqref{eq:change-of-measure},
  \[
    \tcP^{n}\left(\mathbf{1}\right) = g^{-1} \cP^{n}\left(g\right) \le
    m(g)
  \]
  Since
  $-\|\phi\|_{L^\infty} \mathbf{1} \le \phi \le \|\phi\|_{L^\infty}
  \mathbf{1}$ and $\tcP^n$ are positive operators, we
  obtain~\eqref{eq:L-infty-g}.

  For \eqref{eq:L-1-g} 
  assume that $\phi\in C^1$ (otherwise it is clearly satisfied).
  In view of \eqref{eq:P-changed-measure}:
  \begin{equation}\label{eq:tP-estimate}    
  \begin{aligned}
    \|\tcP^{n} \left(\mycenter{(g^{-1} h) \phi}{\tm}\right)\|_{L^1(\tm)}
    & = \| g^{-1} \cP^n(\mycenter{h \phi}{m})- \frac{m(g \phi)}{m(g)} g^{-1}
    \cP^n(\mycenter{g}{m})\|_{L^1(\tm)}
    \\
    & = \|\cP^n(\mycenter{h \phi}{m})- \frac{m(g
      \phi)}{m(g)}\cP^n(\mycenter{g}{m})\|_{L^1(m)}
    \\
    & \le \|\cP^n(\mycenter{h \phi}{m})\|_{L^1(m)} + \left|\frac{m(g
        \phi)}{m(g)}\right| \|\cP^n(\mycenter{g}{m})\|_{L^1(m)}
    \end{aligned}
  \end{equation}
  By \cite[Lemma 2.3]{Nicol:2018aa}, 
  there is an affine function $\mathcal{F}:\R\to \R$ such that for
  $\phi\in C^1([0,1])$ and $h\in \cone$
  can write $\mycenter{\phi h}{m}= \Psi_1-\Psi_2$ with
  $\Psi_{1}, \Psi_{2}\in \cone$ and
  $\|\Psi_{1,2}\|_{L^1(m)}\le \mathcal{F}(\|\phi\|_{C^1}+m(h))$.
  By \cite[Theorem 1.2]{Nicol:2018aa}, for an observable $\psi$ in the cone
  $\cone$ and for any sequence of maps $\cT^\infty$, we have
  \begin{align*}
    \int_X |\cP^{n} (\mycenter{\psi}{m}) | d m \leq C_\alpha
    \|\psi\|_{L^1(m)} n^{-\frac{1}{\alpha}+1}(\log
    n)^{\frac{1}{\alpha}}
  \end{align*}
  where $C_{\alpha}$ depends only on $T_\alpha$.
  Applying these to \eqref{eq:tP-estimate}, we obtain \eqref{eq:L-1-g}.
\end{proof}

\begin{lema}\label{H_n-bound}
  Let $\phi\in C^1$ and $0< \alpha < 1$. Then
  
  \begin{align*}
    \|H_n\circ\cT^n\|_{L^p(m)}  \le  
    \begin{cases}
      C_{\alpha, \|\phi\|_{C^1}+m(g)} (\log n )^{1+\frac{1}{1-\alpha}} &
      \text{ if } 1\le p=\frac{1} {\alpha} -1 \\
      \frac{1}{1-\frac{1}{p}\left({\frac{1}{\alpha}-1}\right)} C_{\alpha,
        \|\phi\|_{C^1}+m(g)} n^{1+\frac{1}{p}(1-\frac{1}{\alpha})} (\log
      n)^{\frac{1}{p\alpha}} & \text{ if } p > \max\{ 1, \frac{1} {\alpha}
      -1\}
    \end{cases}
  \end{align*}
  (the first case is valid for $0< \alpha \le \frac{1}{2}$) and the same
  bounds hold for $\|\tH_n\circ\cT^n\|_{L^p(\tm)}$, where
 \[
   \text{ $H_n\circ\cT^n:= \bE_m(\mycenter{{S}_{n-1}}{m}|\mathcal{B}_n)$,
     $\tH_n\circ\cT^n:= \tbE(\mycenter{{S}_{n-1}}{\tm}|\mathcal{B}_n)$,
     $\mathcal{B}_n:=\cT^{-n}\mathcal{B}$.
   }
 \]

 Note that if $1\le p< \frac{1}{\alpha}-1$, then
 $\|H_n\circ\cT^n\|_{L^p(m)} \le C_{p, \alpha, \|\phi\|_{C^1}+m(g)}$,
 though this observation does not play a role in our subsequent analysis.
\end{lema}

\begin{proof}
  We prove the statement for $\tH_n$. The one for $H_n$ is obtained the
  same way, using Proposition~\ref{nicoldecay2} instead of
  \eqref{eq:L-1-g}.

  Using the definition of $\tH_n$:
  \begin{equation}\label{eq:H_n-sum}
    \|\tH_n\circ\cT^n\|_{L^p(\tm)}=\|\sum_{k=1}^{n-1}\tbE(\mycenter{ \phi
      \circ \cT^k}{\tm}|\mathcal{B}_n)\|_{L^p(\tm)} \le \sum_{k=1}^{n-1}
    \|\tbE(\mycenter{ \phi \circ \cT^k}{\tm}|\mathcal{B}_n)\|_{L^p(\tm)}
  \end{equation}
  We will bound each term of the above sum in both $L^1$ and $L^\infty$,
  and then use \eqref{eq:L-p-interpolation} to obtain an $L^p$-bound.

  In $L^\infty$ we have
  \[
    \|\tbE(\mycenter{ \phi \circ \cT^k}{\tm}|\mathcal{B}_n)\|_{L^\infty(\tm)}
    \le \|\mycenter{ \phi \circ \cT^k}{\tm}\|_{L^\infty(\tm)} \le 2
    \|\phi\|_{L^\infty(\tm)}.
  \]

  In $L^1$ we use~\eqref{eq:conditional-expectation} to compute the
  conditional expectation. Since the conditional expectation preserves the
  expected value, one can check that the centering holds as written
  below\footnote{$\tm (\phi \cdot\tcP^k(\mathbf{1}))= \tm(\phi \circ
    \cT^k)$ because, by the definition of the transfer operator,
    $\int \phi \cdot\tcP^k(\mathbf{1}) d \tm =\int \phi \circ \cT^k \cdot
    \mathbf{1} d \tm$ }. We can then use \eqref{eq:L-1-g} for the decay,
  with $h=\cP^k(g)$, because $\tcP^k(\mathbf{1})=g^{-1} \cP^k(g)$.
  \begin{align*}
    & 
      \|\tbE(\mycenter{ \phi \circ \cT^k}{\tm}|\mathcal{B}_n)\|_{L^1(\tm)} 
      = \|\dfrac{
      \tP_{n}\circ\hdots\circ \tP_{k+1}
      (\mycenter{ \phi \cdot\tcP^k(\mathbf{1})}{\tm}
      )}{\tcP^n(\mathbf{1})}\circ\cT^n\|_{L^1(\tm)}
    \\
    & 
      = \|\tP_{n}\circ\hdots\circ \tP_{k+1}
      (\mycenter{ \phi \cdot\tcP^k(\mathbf{1})}{\tm} )\|_{L^1(\tm)}
      = \|\tP_{n}\circ\hdots\circ \tP_{k+1}
      (\mycenter{ \phi \cdot g^{-1} \cP^k(g)}{\tm} )\|_{L^1(\tm)}
    \\
    &  \le
      C_\alpha \mathcal{F}_1(\|\phi\|_{C^1}+m(\cP^k(g)))
      (n-k)^{-\frac{1}{\alpha}+1} (\log
      (n-k))^{\frac{1}{\alpha}}.
  \end{align*}
  Note that $m(\cP^k(g))=m(g)$, so the coefficient above does not depend on
  $k$.

  Apply now \eqref{eq:L-p-interpolation}, noting that
  $\|f\|_{\infty}^{1-\frac{1}{p}}\le \max\{1,\|f\|_{\infty}\}$, to obtain
  for $1\le p \le \infty$ that
  \[
    \|\tbE(\mycenter{ \phi \circ \cT^k}{\tm}|\mathcal{B}_n)\|_{L^p(\tm)}
    \le C_{ \alpha, \|\phi\|_{C^1}+m(g)}
    \left[(n-k)^{-\frac{1}{\alpha}+1} (\log
      (n-k))^{\frac{1}{\alpha}}\right]^{\frac{1}{p}}
  \]

  If $p=\frac{1}{\alpha}-1 \ge 1$ we bound the last sum in
  \eqref{eq:H_n-sum} by
  $\sum_{k=1}^{n-1} C_{ \alpha, \|\phi\|_{C^1}+m(g)} \left[k^{-1} (\log
    (n))^{\frac{1}{p\alpha}}\right]$ to obtain
  \[
    \|\tH_n\circ\cT^n\|_{L^p(\tm)} \le C_{\alpha, \|\phi\|_{C^1}+m(g)}
    (\log n )^{1+\frac{1}{1-\alpha}}.
  \]

  If $p>\max\{1,\frac{1}{\alpha}-1\}$ we bound the sum in
  \eqref{eq:H_n-sum} by
  $\sum_{k=1}^{n-1}C_{ \alpha, \|\phi\|_{C^1}+m(g)}
  \left[k^{-\frac{1}{\alpha}+1} (\log
    (n))^{\frac{1}{\alpha}}\right]^{\frac{1}{p}}$ to obtain the bound
  \[
    \|\tH_n\circ\cT^n\|_{L^p(\tm)}\le
    \frac{1}{1-\frac{1}{p}\left({\frac{1}{\alpha}-1}\right)} C_{\alpha,
      \|\phi\|_{C^1}+m(g)} n^{1+\frac{1}{p}(1-\frac{1}{\alpha})} (\log
    n)^{\frac{1}{p\alpha}}.
  \]

  Note that if $1\le p<\frac{1}{\alpha}-1$ the series converges to a
  constant $ C_{p, \alpha, \|\phi\|_{C^1}+m(g)}$.
\end{proof}

A useful remark is the following lower bound for functions in the cone
$\cone$:

\begin{prop}[{\cite[Lemma 2.4]{LivSauVai}}]\label{lowerbound}
    For every function $f \in \cone$ one has
    \begin{align*}
        \inf_{x\in [0,1]} f(x) = f(1) \geq \min\left\{a, \left[\dfrac{\alpha(1+\alpha)}{a^\alpha}\right]^{\frac{1}{1-\alpha}} \right\} m(f).
    \end{align*}
    
    Denote the constant in the above expression by $D_{\alpha}$. Then
    $\cP^n \mathbf{1} \geq D_{\alpha} > 0$ for all $n\geq 1$.
\end{prop}

We will also use Rio's inequality, taken from \cite{Merlevede2006aa}. This
is a concentration inequality that allows us to bound the moments of
Birkhoff sums.

\begin{prop}[{\cite{Merlevede2006aa, Rio}}]\label{Rio}
  Let $\{ X_i\}$ be a sequence of $L^2$ centered random variables with
  filtration $\cF_i = \sigma(X_1,\hdots,X_i)$. Let $p\geq 1$ and define
    \begin{align*}
        b_{i,n} = \max_{i\leq u \leq n} \| X_i\sum_{k=i}^u \bE(X_k|\cF_i) \|_{L^p},
    \end{align*}
    then
    \begin{align*}
        \bE|X_1 + \hdots + X_n|^{2p} \leq \left(4p \sum_{i=1}^n b_{i,n} \right)^p.
    \end{align*}
\end{prop}

\section{Polynomial large and moderate deviations estimates}

\subsection{Sequential dynamical systems}

Recall we fixed a sequence
$\cT^\infty = \hdots T_{\alpha_n},\dots, T_{\alpha_1}$
where each of the maps  is of the form
\begin{align*}
    T_{\alpha_j}(x) = \left.
  \begin{cases}
    x + 2^{\alpha_j} x^{1+\alpha_j}, & 0\leq x\leq 1/2, \\
    2x - 1, & 1/2 \leq x \leq 1 
  \end{cases}
  \right.,
\end{align*}
for $0 < \alpha_j\leq \alpha <1$. In the first part of this section we
prove that for such a fixed sequence of maps $\cT^\infty$, a polynomial
large deviations bound holds for the centered sums.

\begin{teo}[Sequential LD]\label{LDseq}
  Let $0 < \alpha <1$ and $\varphi\in\mathcal{C}^1([0,1])$. Then the
  centered sums satisfy the following large deviations upper bound: 
  for any $\epsilon > 0$ and $p > \max\{1, \frac {1} {\alpha} -1 \}$,
    \begin{align*}
      m\left\{ x : \left| \sum_{j=1}^{n}
      [\varphi (\cT^j)(x)- m(\varphi (\cT^j))] \right| > n\epsilon\right\} \leq
      \left(\frac{4p}{1-\frac{1}{p}\left({\frac{1}{\alpha}-1}\right)}\right)^p 
      C^p_{\alpha, \|\varphi\|_{C^1}} 
      n^{1-\frac{1}{\alpha}} (\log n)^{\frac{1}{\alpha}}  \epsilon^{-2p} 
     \end{align*}
     where $C = C_{\alpha,  \|\varphi\|_{C^1}}$ is a constant depending
     on $\alpha$ and the ${C}^1$ norm of $\varphi$, but \emph{not} on
     the sequence $\cT^\infty$.

     In particular, for $p > \max\{1, \frac {1} {\alpha} -1 \}$ we obtain
     the following moment estimate:
     \begin{equation}\label{eq.moment-estimate_LD}
       \bE_{m}\left| \sum_{j=1}^{n}
         [\varphi (\cT^j)(x)- m(\varphi (\cT^j))] \right|^{2p}
       \le
       \left(
         \frac{4p}{1-\frac{1}{p}\left({\frac{1}{\alpha}-1}\right)}\right)^p
       C^p_{\alpha,\|\varphi\|_{C^1}} n^{2p+(1-\frac{1}{\alpha})}(\log
       n)^{\frac{1}{\alpha}}
     \end{equation}

     The same estimates (by the same proof) hold for the measure $\tm$. More precisely,
     \begin{align*}
      \tm \left\{ x : \left| \sum_{j=1}^{n}
      [\varphi (\cT^j)(x)- \tm (\varphi (\cT^j))] \right| > n\epsilon\right\} \leq
      \left(\frac{4p}{1-\frac{1}{p}\left({\frac{1}{\alpha}-1}\right)}\right)^p 
      \widetilde{C}^p_{\alpha, \|\varphi\|_{C^1}} 
      n^{1-\frac{1}{\alpha}} (\log n)^{\frac{1}{\alpha}}  \epsilon^{-2p} 
     \end{align*}
\end{teo}

\begin{rem}
  Our result gives that the dependence on $\epsilon$ is better in the case
  $\alpha>\frac{1}{2}$, where we may take $p\to1$ to obtain
  \begin{align*}
      m\left\{ x : \left| \sum_{j=1}^{n}
      [\varphi (\cT^j)(x)- m(\varphi (\cT^j))] \right| > n\epsilon\right\} \leq
     \tilde{C}_{\alpha, \|\varphi\|_{C^1}} 
      n^{1-\frac{1}{\alpha}} (\log n)^{\frac{1}{\alpha}}  \epsilon^{-2} 
     \end{align*}
     where $\tilde{C}_{\alpha, \|\varphi\|_{C^1}} = \frac{4\alpha}{2\alpha-1}   C_{\alpha, \|\varphi\|_{C^1}} $.
  The worse bound for $\alpha <\frac{1}{2}$  is probably an
  artefact of our proof, and not an optimal result.
\end{rem}

\begin{rem}
  In~\cite[Corollary A.2]{Mel2009}, improving~\cite{Melbourne:2008aa},
  these bounds are shown to be basically optimal if a single map
  $T_{\alpha}$, $0 < \alpha < 1$, is being iterated, with respect to its
  absolutely continuous invariant measure $\mu$: there exists an open and
  dense set of H\"older observables $\varphi$ such that
  \[
    \mu\left\{ x : \sum_{j=1}^{n} [\varphi (\cT^j)(x)- \mu(\varphi
      (\cT^j))]> n\epsilon\right\} \geq C_\epsilon n^{1-\frac{1}{\alpha}}
    \qquad \text{ infinitely often.}
  \]
\end{rem}

As a corollary of Theorem~\ref{LDseq} we obtain moderate deviations
estimates.

\begin{teo}[Sequential Moderate Deviations]\label{MDseq}
  Let $0 < \alpha <1$, $\beta:=\frac{1}{\alpha}-1$,
  $\varphi\in\mathcal{C}^1([0,1])$ and $\tau\in (\frac{1}{2},1]$. Then the
  centered sums satisfy the following moderate deviations upper bounds,
  where $C_{\alpha, \|\varphi\|_{C^1}}$ is a constant depending on $\alpha$
  and the ${C}^1$ norm of $\varphi$, but \emph{not} on the sequence
  $\cT^\infty$:

  (a) If $\alpha >\frac{1}{2}$ then for any $t > 0$
  \begin{align*}
    m\left\{ x : \left|\sum_{j=1}^{n}
    [\varphi (\cT^j)(x)- m(\varphi (\cT^j))] \right| > n^{\tau}t \right\} \le 
    \frac{4}{2-\frac{1}{\alpha}} C_{\alpha,  \|\varphi\|_{C^1}} 
    n^{-\beta+2(1-\tau)} (\log n)^{\frac{1}{\alpha}}  t^{-2}
  \end{align*}
  
  (b) If $\alpha \le \frac{1}{2}$ then
  \begin{align*}
    m\left\{ x : \left| \sum_{j=1}^{n}
    [\varphi (\cT^j)(x)- m(\varphi (\cT^j))] \right|> n^{\tau}t \right\} 
    \le  (4\beta)^{\beta} C_{\alpha, \|\phi\|_{C^1}+m(g)}^{\beta}(\log
    n)^{\frac{2}{\alpha} -1}n^{-\beta(2\tau-1)}t^{-2\beta}
  \end{align*}

  The same estimates (by the same proof) hold for the measure $\tm$.
\end{teo}

\begin{proof}[Proof of Theorem~\ref{LDseq}]
  
  We prove the estimate for $m$, the one for $\tm$ is obtained the same
  way.

  Fix $n$ and for $i\in\{1,\hdots,n \}$, define the sequence of
  $\sigma-$algebras $\cF_{i,n} = \cF_i = \cT^{-(n-i)}(\cB)$. Note that
  $\cF_i \subset \cF_{i+1}$ hence $\{\cF_i\}_{i=1}^n$ is an increasing
  sequence of $\sigma-$algebras. Take
  $X_i = \centerold{\varphi}{n-i}\circ\cT^{n-i} $, so that $X_i$ is $\cF_i$
  measurable. Recall that
  $\psi_j=\centerold{\varphi}{j} +H_j-H_{j+1}\circ T_{j+1}$ for all
  $j\ge 0$. We define $Y_i=\psi_{n-i}\circ\cT^{n-i} $,
  $h_i=H_{n-i}\circ\cT^{n-i}$ for $i\in\{1,\hdots,n \}$. Hence
  $Y_i=X_i + h_i -h_{i-1}$.

  Note also that
  $\cG_i:=\sigma(X_1,\ldots, X_i)\subset \sigma(\cF_1, \ldots,
  \cF_i)=\cF_i$, as $\sigma(X_i)\subset \cF_i$ for all $i$. Since
  $\bE(\psi_i \circ \cT^i | \cT^{-i-1} \cB)=0$, $\bE (Y_i | \cF_{j})=0$ for
  all $i>j$. Hence
  $\bE( Y_i | \cG_j ))= \bE( \bE(Y_i | \cF_j )|\cG_j)=0$ for $i>j$.


  For $p\ge 1$ define $b_{i,n}$ as in Rio's inequality, with $\cG_i$, $X_i$
  as described above so that
  \begin{align*}
    b_{i,n} = \max_{i\leq u \leq n} \left\| X_i\sum_{k=i}^u \bE(X_k|\cG_i) \right\|_{L^p(m)}.
  \end{align*}
  Here all the expectations are taken with respect to $m$.

  Recalling the expression we have for the martingale difference, we can
  write the sum inside the $p$-norm as
\begin{align*}
  \sum_{k=i}^u \bE(X_k|\cG_i)
  &= \sum_{k=i}^u\left[ \bE(Y_k|\cG_i) - \bE(h_k
    |\cG_i) + \bE(h_{k-1}|\cG_i)\right] \\ 
  & = [\sum_{k=i}^u \bE(Y_k|\cG_i) ] + \bE(h_{i-1}|\cG_i)
    - \bE(h_{u}|\cG_i).
\end{align*}
If $k>i$, then $\bE(Y_k|\cG_i)=0$. This reduces the expression above to
\begin{align*}
  \bE(Y_i |\cG_i) + \bE(h_{i-1}|\cG_i) - \bE(h_{u}|\cG_i).
\end{align*}
      
We note that $\|E[f|\cG]\|_p \le \|f\|_p$ for any $f\in L^p (m)$, $p\ge 1$.
Therefore, we may bound $b_{i,n}$ by
$\max_{i\le u\le n} \|X_i \|_{\infty} (\|Y_i \|_p + \|h_{i-1}\|_p +
\|h_{u}\|_p)$.

We now pick $p > \max\{1, \frac {1} {\alpha} -1 \}$. Since
$\|X_i\|_{\infty}$ is uniformly bounded by $2\|\phi\|_{\infty} $ and
$Y_i=X_i + h_i -h_{i-1}$, we may bound
$\max_{i\le u\le n} \|X_i \|_{\infty} (\|Y_i \|_p + \|h_{i-1}\|_p +
\|h_{u}\|_p)$ by
\[
  \frac{1}{1-\frac{1}{p}\left({\frac{1}{\alpha}-1}\right)} C_{\alpha,
    \|\varphi\|_{C^1}} n^{1+\frac{1}{p}(1-\frac{1}{\alpha})} (\log
  n)^{\frac{1}{p \alpha}}
\]
where $C_{\alpha, \|\varphi\|_{C^1}}$ is independent of $n$. This is a
consequence of Proposition~\ref{H_n-bound}.
  
Therefore
$(4p\sum_{i=1}^n b_{i,n})^p\le \left(
  \frac{4p}{1-\frac{1}{p}\left({\frac{1}{\alpha}-1}\right)}\right)^p
C^p_{\alpha,\|\varphi\|_{C^1}} n^{2p+(1-\frac{1}{\alpha})}(\log
n)^{\frac{1}{\alpha}}$, which, by Rio's inequality (see
Proposition~\ref{Rio}), is an upper bound for
$\bE_{m}|X_1+X_2+\cdots+ X_n|^{2p}$; this
proves~\eqref{eq.moment-estimate_LD}.
Thus, by Markov's inequality,
\begin{align*}
  m(|X_1+\hdots+X_n |^{2p} > n^{2p}\epsilon^{2p})  &
   \leq \left(\frac{4p}{1-\frac{1}{p}\left({\frac{1}{\alpha}-1}\right)}\right)^p 
   C_{\alpha,\|\varphi\|_{C^1}}^p  (n^{-2p}\epsilon^{-2p})
    n^{2p+(1-\frac{1}{\alpha})}(\log n)^{\frac{1}{\alpha}}  \\
  &
  = \left(\frac{4p}{1-\frac{1}{p}\left({\frac{1}{\alpha}-1}\right)}\right)^p
    C_{\alpha,\|\varphi\|_{C^1}}^p
    n^{1-\frac{1}{\alpha}}(\log n)^{\frac{1}{\alpha}} \epsilon^{-2p}
\end{align*}
which is the claimed Large Deviation bound.
\end{proof}

\begin{proof}[Proof of Theorem~\ref{MDseq}]

  Assume the hypotheses of Theorem~\ref{MDseq} and let
  $\tau\in (\frac{1}{2},1]$.

  (a) Let $\alpha >\frac{1}{2}$ so that $\frac{1}{\alpha}-1<1$. For
  $\tau \in (\frac{1}{2},1]$ define $t n^{\tau}=n\epsilon$ so that
  $\epsilon=t n^{\tau-1}$. Then by Theorem~\ref{LDseq} for any $t > 0$ and
  $p > 1$,
    \begin{align*}
      m\left\{ x : \left|\sum_{j=1}^{n}
      [\varphi (\cT^j)(x)- m(\varphi (\cT^j))]\right| > n^{\tau} t \right\} \leq
      \left(
      \frac{4p}{1-\frac{1}{p}\left({\frac{1}{\alpha}-1}\right)}\right)^p
      C^p_{\alpha, \|\varphi\|_{C^1}} 
      n^{1-\frac{1}{\alpha}} (\log n)^{\frac{1}{\alpha}}  n^{2 p(1-\tau)}t^{-2p}
     \end{align*}
     where $C_{\alpha, \|\varphi\|_{C^1}}$ is a constant depending on
     $\alpha$ and the ${C}^1$ norm of $\varphi$, but \emph{not} on the
     sequence $\cT^\infty$ or $p$. Fix $t>0$ and let $p\to 1$ to obtain,
     where $\beta:=\frac{1}{\alpha}-1$,
  \begin{align*}
    m\left\{ x : \left| \sum_{j=1}^{n}
    [\varphi (\cT^j)(x)- m(\varphi (\cT^j))]\right|> n^{\tau} t \right\} \leq
    \frac{4}{2-\frac{1}{\alpha}} C_{\alpha, \|\varphi\|_{C^1}}
    n^{-\beta+2(1-\tau)} (\log n)^{\frac{1}{\alpha}}  t^{-2}
  \end{align*}
     
  (b) If $\alpha \le \frac{1}{2}$
  we take $p= \frac{1}{\alpha}-1\ge 1$ and have from Lemma~\ref{H_n-bound}
  that
  $\|H_n\circ\cT^n\|_{L^p(m)} \le C_{\alpha, \|\phi\|_{C^1}+m(g)} (\log n
  )^{1+\frac{1}{1-\alpha}}$. In the proof of Theorem~\ref{LDseq} we can
  then bound
  $(4p\sum_{i=1}^n b_{i,n})^p \le (4p)^p C_{\alpha, \|\phi\|_{C^1}+m(g)}^p
  n^p (\log n)^{p+\frac{p}{1-\alpha}}$ and hence by Rio's inequality
  \[
    \bE_m|X_1+\ldots + X_n|^{2p}\le (4p)^p C^p_{\alpha,
      \|\phi\|_{C^1}+m(g)} n^p (\log
    n)^{p+\frac{p}{1-\alpha}}.
  \]
  Markov's inequality gives
     \[
       m(|X_1+ \ldots +X_n|>n \epsilon )\le n^{-2p}\epsilon^{-2p}(4p)^p
       C^p_{\alpha, \|\phi\|_{C^1}+m(g)} n^p (\log
       n)^{p+\frac{p}{1-\alpha}}
     \]
     \[
       =n^{-p}\epsilon^{-2p}(4p)^p C^p_{\alpha, \|\phi\|_{C^1}+m(g)} (\log
       n)^{p+\frac{p}{1-\alpha}}.
   \]
   Taking $n\epsilon =n^{\tau}t$ for $\tau \in (\frac{1}{2},1]$ and the
   choice $p=\beta=\frac{1}{\alpha}-1$ we obtain
   \[
     m(|X_1+ \ldots +X_n|>n^{\tau} t ) \le (4\beta)^{\beta} C_{\alpha,
       \|\phi\|_{C^1}+m(g)}^{\beta}(\log n)^{\frac{2}{\alpha}
       -1}n^{-\beta(2\tau-1)}t^{-2\beta}
   \]
   as claimed.
 \end{proof}
 
\subsection{Random dynamical systems}

Now we prove large deviations estimates for the randomized systems. First we recall some notation. The annealed transfer operator $P \colon L^{1}(m) \rightarrow L^{1}(m)$ is defined by averaging over all the transformations:
\begin{align*}
  P=\sum_{\beta \in \Omega} p_{\beta} P_{\beta} = \int_{\Sigma} P_\omega \ d\nu(\omega).
\end{align*}
This operator is dual to the annealed Koopman operator
$U \colon L^{\infty}(m) \rightarrow L^{\infty}(m)$ defined by
\begin{align*}
  (U\varphi) (x)=\sum_{\beta \in \Omega} p_{\beta} \varphi (T_{\beta} x)=
  \int_{\Sigma}  \varphi (T_{\omega} x)  d\nu(\omega)  =  \int_{\Sigma}
  \tilde\varphi (F(\omega, x))  d\nu(\omega)
\end{align*}
where $\tilde{\phi}(\omega, x) := \phi(x)$. The annealed operators satisfy
the duality relationship
\begin{align*}
  \int_X (U\varphi) \cdot \psi \ dm = \int_X \varphi \cdot P\psi \ dm
\end{align*}
for all observables $\varphi\in L^\infty(m)$ and $\psi\in L^1(m)$.

\begin{rem}\label{stationary-measure}
  It is easy to see that the averaged transfer operator $P$ has no worse
  rate of decay in $L^1$ than the slowest of the maps (so better than
  $n^{-\frac{1}{\alpha}+1}(\log n)^{\frac{1}{\alpha}}$, by
  Proposition~\ref{nicoldecay2}). By taking a limit point of
  $\frac{1}{n}\sum_{k=1}^n P^k(\one)$, there is an invariant vector $h$ for
  $P$ in the cone $\cone$, see~\cite{LivSauVai}. The measure $\mu = h m$ is
  stationary for the RDS; by Proposition~\ref{lowerbound},
  $h\ge D_{\alpha}>0$.

  Moreover, Bahsoun and Bose~\cite{Bahsoun:2016ab, Bahsoun:2016aa} have
  shown that there exists a unique absolutely continuous (with respect to
  the Lebesgue measure) stationary measure $\mu$, and $\nu\otimes\mu$ is
  mixing --- so also ergodic.
 
\end{rem}

Using the same idea as in the proof of Theorem~\ref{LDseq}, we can obtain
an annealed result for the random dynamical system. Note that $P_{\mu}$,
the transfer operator with respect to the stationary measure $\mu$,
satisfies $P_{\mu}\mathbf{1}=\mathbf{1}$ and so
$\|P_{\mu}\phi\|_{\infty} \le P_{\mu} (\|\phi\|_{\infty} ) =
\|\phi\|_{\infty} \|P_{\mu} \mathbf{1}\|_{\infty} = \|\phi\|_{\infty}$. An
easy calculation shows that $P_{\mu} (\phi)=\frac{1}{h} P (h \phi)$ where
$h\in \cone$ is the density of the invariant measure $\mu$ and hence
$h \ge D_{\alpha}m(h)$ is bounded below. As before this observation allows
us to bootstrap in some sense the $L^1({\mu})$ decay rate to $L^p({\mu})$
for $p \ge 1$, a technique used in~\cite{Melbourne:2008aa,Mel2009}.

\begin{teo}[Annealed LD]\label{annealedLD}
  Let $\varphi\in\mathcal{C}^1([0,1])$ with $\mu(\varphi) = 0$ and let
  $0 < \alpha <1$. Then the Birkhoff averages have annealed large
  deviations with respect to the measure $\nu\otimes\mu$ with rate
    \begin{align*}
      (\nu\otimes\mu)\{(\omega,x):\left|\sum_{j=1}^{n} \varphi \circ
      \cT^j_{\omega}(x)\right| \geq n\epsilon \} \leq 
      C_{\alpha, p, \|\varphi\|_{C^1}} 
      n^{1-\frac{1}{\alpha}} (\log n)^{\frac{1}{\alpha}}  \epsilon^{-2p} 
    \end{align*}
    for any $p > \max\{1, \frac {1} {\alpha} -1 \}$.

    Note that the Birkhoff sums above \emph{are not centered for a given
      realization $\omega$}, only on average over $\Sigma$.
\end{teo}

\begin{proof}
  To prove this result we will use the construction used to prove the
  annealed CLT in \cite{Aimino:2015aa}: let
  $\Omegastar := X^{\mathbb{N}_0}$, endowed with the $\sigma$-algebra $\cG$
  generated by the cylinders, and the left shift operator
  $\deltatau \colon \Omegastar\to\Omegastar$.

  Denote by $\pi$ the projection from $\Omegastar$ onto the 0-th
  coordinate, that is, $\pi(x) = x_0$ for $x = (x_0,x_1,\hdots)$. We can
  lift any observable $\varphi\colon X\to \mathbb{R}$ to an observable on
  $\Omegastar$ by setting
  $\phi_\pi := \varphi\circ \pi\colon \Omegastar\to\mathbb{R}$.

  Following~\cite[\S 4]{Aimino:2015aa}, one can introduce a
  $\deltatau$-invariant probability measure $\mu_c$ on $\Omegastar$ such
  that
  
  $\bE_\mu(\varphi) = \bE_{\mu_c}(\varphi_\pi)$, and the law of $S_n(\phi)$
  on $\Sigma\times X$
  under $\nu\otimes\mu$ is the same as the law of the $n$-th Birkhoff sum
  of $\varphi_\pi$ on $\Omegastar$ under $\mu_c$ and $\deltatau$; thus it
  suffices to establish large deviations for the latter.

  Define now
  \begin{align*}
    H_n := \sum_{k=1}^n P_{\mu}^{k}(\varphi):X \to \R
  \end{align*}
  From the relation $P_{\mu} (.)=\frac{1}{h} P (.h)$, we have that
  $\|P_{\mu}^n (\varphi )\|_{L^1(\mu)} \le C_{\alpha,\varphi}
  n^{1-\frac{1}{\alpha}}(\log n)^{1/\alpha}$ because $\mu(\phi)=0$.
  %
  %
  We calculate
  $E_{\mu} |P_{\mu}^i (\varphi )|^p = E_{\mu} [|P_{\mu}^i (\varphi
  )|^{p-1}|P_{\mu}^i (\varphi )|]\le \|P_{\mu}^i (\varphi
  )\|_{\infty}^{p-1} \|P_{\mu}^i (\varphi )\|_{L^1(\mu)}$. Hence
  $\|P_{\mu}^k (\varphi )\|_{L^p (\mu)} \le C k^{(1-1/\alpha)/p} (\log
  k)^{1/(p\alpha)}$ and thus $\|H_n \|_{L^p (\mu)}$ satisfies the bounds of
  Lemma~\ref{H_n-bound}.

  We lift $\varphi$ and $H_n$ to $\Omegastar$ and denote them by
  $\varphi_\pi$ and $H_{n,\pi}$ respectively, and define
  \begin{align*}
    \chi_n := \varphi_\pi + H_{n,\pi} - H_{n,\pi}\circ \deltatau:
    \Omegastar \to \R.
  \end{align*}
  
  We now continue as in the proof of Theorem~\ref{LDseq}, applying Rio's
  inequality. 
  For $i=1,\ldots,n$ take the sequences
  $\{X_i = \varphi_\pi \circ \deltatau^{n-i}\}$,
  $\{ Y_i = \chi_{n-i} \circ \deltatau^{n-i}\}$ and
  $\cG_i = \deltatau^{-(n-i)}\cG$. We have $\bE_{\mu_c} [Y_i|\cG_k]=0$ for
  $i>k$ and so, for $p > \max\{1, \frac {1} {\alpha} -1 \}$,
  \begin{align*}
    b_{i,n} =  \max_{i\leq u\leq n} \left\| X_i \sum_{k=i}^u \bE_{\mu_c}(X_k |
    \cG_i)  \right\|_{L^p(\mu_c)} \le C  n^{1+\frac{1}{p}(1-\frac{1}{\alpha})}
    (\log n)^{\frac{1}{p \alpha}}
  \end{align*}
  which gives, as in Theorem~\ref{LDseq},
  \begin{align*}
    \mu_c( |X_1+\hdots+X_n |^{2p} > n^{2p}\epsilon^{2 p}) \leq
    C_{\alpha, \varphi, p} n^{1-\frac{1}{\alpha}}(\log
    n)^{\frac{1}{\alpha}} \epsilon^{-2p}
  \end{align*}
\end{proof}

Using similar ideas, it is possible to obtain an annealed central limit
theorem. This has been established already by Young Tower techniques
in~\cite[Theorem 3.2]{Bahsoun:2016aa}. We include the statement of the
annealed central limit and an alternative proof for completeness and to
give an expression for the annealed variance.

\begin{prop}[Annealed CLT]\label{annealed_CLT}
  If $\alpha <\frac{1}{2}$ and $\phi\in C^1$ with $\mu(\varphi)=0$ then a
  central limit theorem holds for $S_{n}\varphi$ on $\Sigma \times X$ with
  respect to the measure $\nu\otimes\mu$, that is,
  $\frac{1}{\sqrt{n}} S_{n}\varphi$ converges in distribution to
  $\mathcal{N}(0,\sigma^2)$, with variance $\sigma^2$ given by
  \[
    \sigma^2=-\mu (\varphi^2) +2 \sum_{k=0}^{\infty}\mu(\varphi U^k
    \varphi)
  \]
\end{prop}

\begin{proof}
  We will use the results of~\cite[Section 4]{Aimino:2015aa} and
  \cite[Theorem 1.1]{Liverani} (see Theorem~\ref{th:liverani} in the
  Appendix). We proceed as in Theorem~\ref{annealedLD}, using the averaged
  operators $U$ and $P$. As in~\cite[Section 4]{Aimino:2015aa}, to $U$
  corresponds a transition probability on $X$ given by
  $U(x,A)=\sum_{\beta} \{p_{\beta}: T_{\beta} x\in A\}$. The stationary
  measure $\mu$ is invariant under $U$. Extend $\mu$ to the unique
  probability measure $\mu_c$ on
  $\Omegastar:=X^{N_0}=\{ \underline{x}=(x_0,x_1,x_2,\ldots,x_n,\ldots)\}$,
  endowed with the $\sigma$-algebra $\mathcal{G}$ given by cylinder sets,
  such corresponding to $\mu$ such that $\{x_n\}_{n\ge 0}$ is a Markov
  chain on $(\Omegastar,\mathcal{G},\mu_c)$ (where $x_n$ is the $n$-th
  coordinate of $\underline{x}$) induced by the random dynamical system.
  The left shift $\tau$ on $\Omegastar$
  preserves $\mu_c$. Given $\varphi: X\to \R$, $\mu(\varphi)=0$, we define
  $\varphi_{\pi}$ on $\Omegastar$ by
  $\varphi_{\pi}(x_0,x_1,x_2,\ldots,x_n,\ldots):=\varphi (x_0)$. As
  in~\cite[Section 4]{Aimino:2015aa}, to prove the CLT for $S_n (\varphi)$
  with respect to $\nu \otimes \mu$ on $\Sigma \times X$ it suffices to
  prove the CLT for the Birkhoff sum
  $\sum_{j=0}^n \varphi_{\pi}\circ \tau^k$ with respect to $\mu_c$ on
  $\Omegastar$.

  We introduce the Koopman operator $\widetilde{U}$ and transfer operator
  $\widetilde{P}$ for the map $\tau$ on the probability space
  $(\Omegastar,\mathcal{G},\mu_c)$. We define the decreasing sequence of
  $\sigma$-algebras $\mathcal{G}_k=\tau^{-k} \mathcal{G}$, and note that
  $\widetilde{P}$, $\widetilde{U}$ satisfy
  $\widetilde{P}^k\widetilde{U}^kf=f$ and
  $\widetilde{U}^k\widetilde{P}^k f=\bE_{\mu_c}(f |\mathcal{G}_k )$ for
  every $\mu_c$-integrable $f$. We note that
  $\varphi_{\pi} \in L^{\infty} (\mu_c)$. As in~\cite[Lemma
  4.2]{Aimino:2015aa} we have
  $\widetilde{P}^n(\varphi_{\pi})=(P^n\varphi)_{\pi}$. Thus
  $\sum_{k=0}^{\infty}\widetilde{P}^k\varphi_{\pi}$ converges in
  $L^1 (\mu_c)$ if $\alpha <\frac{1}{2}$
  and therefore
  $\sum_{k=0}^{\infty} | \int \varphi_{\pi} \widetilde{U}^k
  \varphi_{\pi}d\mu_c | <\infty$.
  Thus the result for $\sum_{j=0}^n \varphi_{\pi}\circ \tau^k$ follows
  from~\cite[Theorem 1.1]{Liverani}. The stated formula for $\sigma^2$ is
  also given in~\cite[Theorem~1.1]{Liverani}.
\end{proof}

We will use the annealed and sequential results to obtain quenched large
deviations for random systems of intermittent maps. We denote the Birkhoff
sums by $S_{n,\omega}(x)$ to stress the dependence on the realization
$\omega$.

\newcommand{\LDdecay}{n^{1-\frac{1}{\alpha}} (\log n)^{\frac{1}{\alpha}}
  \epsilon^{-2p}}

\begin{teo}[Quenched LD]\label{th:quenched LD}
  Suppose $\varphi \in \mathcal{C}^1$ and $\mu (\varphi)=0$. Fix
  $0< \alpha < 1$. Then, given $p > \max\{1, \frac {1} {\alpha} -1 \}$ and
  $\kappa := \lceil \frac{4p}{1-\alpha}\rceil$ (rounded up), for $\nu$-almost every
  realization $\omega\in \Sigma$ the Birkhoff averages have large
  deviations with polynomial rate, \emph{even without centering}: there is
  an $N(\omega)$ such that for each $\epsilon>0$
  \begin{align*}
    m \{ x: |S_{n,\omega}\varphi |> 4 n\epsilon\} \leq  
    C_{\alpha, p, \varphi} n^{1-\frac{1}{\alpha}} (\log n)^{\frac{1}{\alpha}}
    \epsilon^{-\kappa}
    \text{ for $n\ge N(\omega)$}.
  \end{align*}

  Note that the Birkhoff sums $S_{n,\omega}\varphi$ above \emph{are not
    centered with respect to the realization $\omega$}, only on average
  over $\Sigma$.
\end{teo}

\begin{rem}
  The point of the above Theorem, compared to the sequential
  Theorem~\ref{LDseq}, is that for almost each realization the large
  deviation estimates hold \emph{even without centering}. That is, the
  contribution of the means (with respect to the measure $m$ on $X$) can be
  ignored for almost each realization $\omega$.
\end{rem}

\begin{proof}[Proof of Theorem \ref{th:quenched LD}]

  Choose $p > \max\{1, \frac {1} {\alpha} -1 \}$ and $\epsilon>0$. By
  Theorem~\ref{LDseq}, for all $\omega \in \Sigma$,
  \begin{equation*}
    m \left\{ x: \left|\frac{1}{n} S_{n,\omega} \varphi (x)
        -\frac{1}{n}\sum_{j=1}^n m(\varphi \circ T^j_{\omega}) \right| \ge
      \epsilon \right\} \le C_{\alpha, p, \varphi} \LDdecay
  \end{equation*}
  with $C_{\alpha,\varphi, \delta}$ independent of $\omega$. Integrating
  over $\Sigma$ with respect to $\nu$ we obtain
\[
  \nu \otimes m \left\{ (\omega, x) : \left| \frac{1}{n}S_{n,\omega}
      \varphi (x) -\frac{1}{n}\sum_{j=1}^n m(\varphi \circ T^j_{\omega})
    \right| \ge \epsilon\right\} \le C_{\alpha, p, \varphi} \LDdecay
\]
By Theorem~\ref{annealedLD}, we also have the annealed estimate for the
non-centered sums:
\[
  \nu \otimes m \left\{ (\omega, x) : \left| \frac{1}{n} S_{n,\omega}
      \varphi (x) \right| \ge \epsilon\right\} \le C_{\alpha, p, \varphi}
  \LDdecay
\]
Theorem~\ref{annealedLD} refers to the measure $\nu \otimes \mu$ but since
$\frac{dm}{d\mu}=\frac{1}{h} \le \frac{1}{D_{\alpha}}$, the large
deviations estimate applies also to $\nu \otimes m$. Observe now that
\[
  \left\{ (\omega, x) : \left| \frac{1}{n}\sum_{j=1}^n m(\varphi \circ
      T^j_{\omega})\right| > 2\epsilon \right\}
\]
\[
  \subset \left\{ (\omega, x) : \left| \frac{1}{n}S_{n,\omega} \varphi
      (x)\right| < \epsilon, \left| \frac{1}{n}S_{n,\omega} \varphi (x)
      -\frac{1}{n}\sum_{j=1}^n m(\varphi \circ T^j_{\omega}) \right| \ge
    \epsilon\right\}
\]
\[
  \bigcup \left\{ (\omega, x) : \left| \frac{1}{n} S_{n,\omega} \varphi (x)
    \right| > \epsilon \right\}.
\]
Thus
\[
  \nu \otimes m \left\{ (\omega, x) : \left|\frac{1}{n} \sum_{j=1}^n
      m(\varphi \circ T^j_{\omega})\right| > 2\epsilon \right\} \le
  K_{\alpha, p, \varphi} \LDdecay
\]
and, as there is no dependence on $x\in X$, this means
\begin{align}\label{eq.quenchedineq}
  \nu  \left\{ \omega  : \left| \frac{1}{n}\sum_{j=1}^n m(\varphi \circ
  T^j_{\omega})\right| > 2\epsilon \right\} \le 
  K_{\alpha, p, \varphi} \LDdecay
\end{align}

Denote $\beta:=\frac{1}{\alpha}-1>0$.

  The proof we give does not give an optimal value of $\kappa$.  In the case  $\beta>1$ a simpler proof may be given 
  but the resulting exponent $\kappa$ is also not optimal and no better than the estimate  we give.
  
   Let $\tau=\frac{2}{\beta}$
   and $\delta>0$ small. Choose
   $\gamma = \frac{1}{2p}(\beta
   -\frac{1}{\tau})-\delta=\frac{\beta}{4p}-\delta$ and
   $\kappa = \lceil (1+\beta^{-1})(4p)\rceil =
   \lceil \frac{4p}{1-\alpha}\rceil $. The notation $\lceil x \rceil$ indicates
   the smallest integer greater than or equal to $x$.
   %
   Then $(2p\gamma-\beta)\tau<-1$ and $\gamma \kappa >\beta$ for
   $\delta>0$ small enough.
  
   For $\epsilon=n^{-\gamma}$ the bound \eqref{eq.quenchedineq} becomes
  \[
    \nu \left\{ \omega : \left| \frac{1}{n}\sum_{j=1}^n m(\varphi \circ
        T^j_{\omega})\right| > 2n^{-\gamma} \right\} \le K_{\alpha, p,
      \varphi} n^{ 2p\gamma} n^{-\beta} (\log n)^{\frac{1}{\alpha}}
  \]

  Consider the subsequence $n_k:=k^{\tau}$. As $(2p\gamma-\beta)\tau<-1$,
  for $\nu$ almost every $\omega$ there exists an $N(\omega)$ such that for
  all $n_k > N(\omega)$,
  \[
    \left| \frac{1}{n_k}\sum_{j=1}^{n_k} m(\varphi \circ
      T^j_{\omega})\right| \le 2n_{k}^{-\gamma}
  \]
  
  If $n_k \le n < n_{k+1}$ then
  \[
    \left| \frac{1}{n}\sum_{j=1}^{n} m(\varphi \circ T^j_{\omega})\right|
    \le \frac{1}{n_k} \left| \sum_{j=1}^{n_k} m(\varphi \circ T^j_{\omega})
      + \sum_{j=n_k+1}^n m(\varphi \circ T^j_{\omega}) \right|
  \]
  \[
    \le 2n_{k}^{-\gamma} +\frac{\|\phi\|_{\infty}}{n_k} | n_{k+1}-n_{k} |
  \]
  There is $K > 0$, independent of $\omega$, depending only on $\tau$,
  $\gamma$ and $\|\phi\|_{\infty}$, such that
  \[
    2n_{k}^{-\gamma} +\frac{\|\phi\|_{\infty}}{n_k} | n_{k+1}-n_{k} | < 3
    n^{-\gamma} \text{\qquad if $k\ge K$.}
  \]
  Indeed, $\lim_{k\to \infty} \frac{n_{k+1}}{n_k}=1$,
  $\frac{1}{n_k} | n_{k+1}-n_{k} | = O(\frac{1}{k})$,
  $\frac{1}{k}=O(\frac{1}{n^{1/\tau}})$ and $ n^{-1/\tau} < n^{-\gamma}$
  because $1/\tau > \gamma$.

  Increase $N(\omega)$ such that $n>N(\omega)$ implies $n \ge K^\tau$ and
  $C_{\alpha,p, \phi}n^{\gamma \kappa -\beta} (\log n)^{1/\alpha}>1$.

  We will show that for $n>N(\omega)$
  \[
    m( x : | \frac{1}{n} S_{n,\omega} \phi (x)| \ge 4\epsilon)\le
    C_{\alpha, p, \phi} \epsilon^{-\kappa} n^{-\beta} (\log n)^{1/\alpha}.
  \]
 
  Suppose $\epsilon< n^{-\gamma}$. Then
  $ C_{\alpha,p, \phi} \epsilon^{-\kappa} n^{-\beta} (\log n)^{1/\alpha}\ge
  C_{\alpha,p, \phi} n^{\gamma \kappa -\beta} (\log n)^{1/\alpha}>1$ and there
  is nothing to prove.

  If $\epsilon \ge n^{-\gamma}$ and $n>N(\omega)$ then, as  $|\frac{1}{n} \sum_{j=1}^n m(\phi \circ T^j_{\omega})| < 3 \epsilon$,
  \[
    \left\{x : | \frac{1}{n} S_{n,\omega} \phi (x)| \ge 4\epsilon \right\}
    \subset \left\{ x : | \frac{1}{n} S_{n,\omega} \phi (x) - \frac{1}{n}
      \sum_{j=1}^n m(\phi \circ T^j_{\omega}) |\ge \epsilon \right\}
   \]

   Hence the result holds by Theorem~\ref{LDseq}, as
   \[
     m( x : | \frac{1}{n} S_{n,\omega} \phi (x) - \frac{1}{n} \sum_{j=1}^n
     m(\phi \circ T^j_{\omega}) |\ge \epsilon) \le C_{\alpha, p, \phi}
     \epsilon^{-2p} n^{-\beta} (\log n)^{1/\alpha}
  \]
  and $2p<\kappa$.
\end{proof}

We remark that the methods used to prove these results in the uniformly
expanding case are not applicable here, as they rely on the
quasi-compactness of the transfer operator. In the uniformly expanding case, which has exponential large deviations for 
H\"older observables, 
it is possible to obtain a rate function.

\section{The Role of Centering in the Quenched CLT for RDS}

In this section we discuss two results: Proposition~\ref{extend}, that the
quenched variance is the same for almost all realizations
$\omega\in \Sigma$, and Theorem~\ref{th.centering}, that generically one
must center the observations in order to obtain a CLT (as opposed to LD
Theorem~\ref{th:quenched LD}, where centering did not affect the quenched
LD). Note that these hinge on the rate of growth of the mean of the
Birkhoff sums; we see that it is $o(n)$ but not $o(\sqrt{n})$. We use the
recent paper by Hella and Stenlund~\cite{2020HellaStenlund} to extend and
clarify results of~\cite{Nicol:2018aa}.

In~\cite[Theorem 3.1]{Nicol:2018aa} a self-norming quenched CLT is obtained
for $\nu$-a.e. realization $\omega$ of the random dynamical system of
Theorem~\ref{annealedLD}. More precisely, recalling the definition of the
centered observables
$\centerold{\varphi}{k}(\omega,x)=\varphi (x) -m(\varphi\circ
\cT^k_{\omega})$ and
$\sigma^2_n (\omega):= \int \left[\sum_{k=1}^n \centerold{\varphi}{k}
  (\omega,\cT^k_\omega x)\right]^2 dx$ it is shown that
$\frac{1}{\sigma_n (\omega)}\sum_{k=1}^n \centerold{\varphi}{k}
(\omega,\cdot)\circ\cT^k_\omega\to N(0,1)$ provided
$\sigma_n^2 \approx n^{\beta}$, with $\alpha <\frac{1}{9}$ and
$\beta >\frac{1}{2(1-2\alpha)}$. Various scenarios under which
$\sigma_n^2 (\omega) > n^{\beta}$ are given in~\cite{Nicol:2018aa}. See
also~\cite{2019HellaLeppanen}.

If the maps $T_{\omega_i}$ preserved the same invariant measure then it
suffices to consider observables with mean zero, since the mean would be
the same along each realization.
In the setting of~\cite{ALS} this is the case, namely all realizations
preserve Haar measure, and the authors address the issue of whether the
variance $\sigma^2_n (\omega)$ can be taken to be the ``same'' for almost
every quenched realization in the setting of random toral automorphisms.
They show that for almost every quenched realization the variance in the
quenched CLT may be taken as a uniform constant. The technique they use is
adapted from random walks in random environments and consists in analyzing
a random dynamical system on a product space.

A natural question is whether in our setup of random intermittent maps,
after centering, $\sigma_n (\omega)$ can be taken to be ``uniform'' over
$\nu$-a.e. realization. Recent results of Hella and
Stenlund~\cite{2020HellaStenlund} give conditions under which
$\frac{1}{n}\sigma^2_n (\omega)\to \sigma^2$ for $\nu$-a.e. $\omega$, as
well as information about rates of convergence. Note that this is also true
in the context of uniformly expanding maps considered
by~\cite{Abdelkader:2016aa} using the same method used
in~\cite{2020HellaStenlund}.

A related question is whether we need to center at all. For example, if
$\mu (\varphi)=0$, where $\mu$ is the stationary measure on $X$, then for
$\nu$-a.e $\omega$

\[
  \lim_{n\to \infty} \frac{1}{n} \sum_{j=1} ^n [\varphi (\cT^j_{\omega}
  x)-m (\varphi (\cT^j_{\omega}))] \to 0 \qquad \text{for $\mu$-a.e. $x$}
\]
by the ergodicity of $\nu\otimes \mu$, but
also
\[
  \lim_{n\to \infty} \frac{1}{n} \sum_{j=1}^n m(\varphi
  (\cT_{\omega}^j))\to 0 \qquad \text{for $\nu$-a.e. $\omega$},
\] 
by the proof of Theorem~\ref{th:quenched LD}. So for the strong law of
large numbers centering is not necessary. Using ideas
of~\cite{Abdelkader:2016aa} we consider the related question of whether
centering is necessary to obtain a quenched CLT with almost surely constant
variance. We show the answer to this is positive: to obtain an almost
surely constant variance in the quenched CLT we need to center.

\subsection{Non-random quenched variance}

For Proposition~\ref{extend}, we verify that our system satisfies the
conditions SA1, SA2, SA3 and SA4 of~\cite{2020HellaStenlund}; then,
by~\cite[Theorem 4.1]{2020HellaStenlund}, the quenched variance is almost
surely the same, equal to the annealed variance.


\begin{prop}\label{extend}
  Let $\alpha<\frac{1}{2}$, $\phi\in C^1$ and define the annealed
  variance
  \begin{align*}
    \sigma^2:=\lim_{n\to\infty} \frac{1}{n}\|\mycenter{S_n}{\nu\otimes
    m}\|_{L^2(\nu\otimes m)}^2= \lim_{n\to\infty} \frac{1}{n}\|S_n -
    \int_{\Sigma\times X} S_n d\; \nu\otimes m\|_{L^2(\nu\otimes m)}^2 
    \\
    =  \sum_{k=0}^{\infty}(2-\delta_{0k})\lim_{i\to\infty} \int_\Sigma
    [m(\varphi_i\varphi_{i+k})-m(\varphi_i)m(\varphi_{i+k})]d\nu
  \end{align*}

  If $\sigma^2>0$ then for $\nu$-a.e. $\omega$
  \[
    \lim_{n\to \infty} \frac{1}{\sqrt{n}} \sum_{j=1}^n \mycenter{\varphi
      (\cT_\omega^j \cdot)}{m} \to^d N(0,\sigma^2)
  \]
  in distribution with respect to $m$.
\end{prop}

\begin{rem}
  Proposition~\ref{annealed_CLT} shows that the annealed CLT holds for
  $\alpha<\frac{1}{2}$ and under the usual genericity conditions the
  annealed variance satisfies $\sigma^2>0$. Thus Proposition~\ref{extend}
  extends~\cite[Theorem 5.3]{Nicol:2018aa} from the parameter range
  $\alpha <\frac{1}{9}$ to $\alpha <\frac{1}{2}$. Note
  that~\cite{2019HellaLeppanen}, proved the CLT for $\alpha <\frac{1}{3}$.
\end{rem}

\begin{proof}[Proof of Proposition~\ref{extend}]
  We will verify conditions SA1, SA2, SA3 and SA4 of~\cite[Theorem
  4.1]{2020HellaStenlund} in our setting, with
  $\eta(k)=C k^{-\frac{1}{\alpha}+1}(\log k)^{\frac{1}{\alpha}}$ in the
  notation of~\cite{2020HellaStenlund}.

  \noindent \textbf{SA1}: If $j>i$ then
 \[\left|\int \varphi\circ \cT^i_{\omega} (x) \varphi\circ \cT^j_{\omega}
     (x) dm-\int \varphi\circ \cT^i_{\omega} (x)dm \int \varphi\circ
     \cT^j_{\omega} (x)dm \right|
\]
\[
  =\left|\int \varphi\circ \cT^{j-i+1}_\omega (\cT^i_{\omega} x) \varphi
    (x) P^i_{\omega} \one dm-\int \varphi \cP^i_{\omega} \one dm \int
    \varphi (x)\cP^j_{\omega} \one dm\right|\le
  C(j-i)^{-\frac{1}{\alpha}+1}(\log (j-i))^{\frac{1}{\alpha}}
\]
by the same argument as in the proof of ~\cite[Proposition
1.3]{Nicol:2018aa}.

\noindent \textbf{SA2}: Our underlying shift $\sigma: \Sigma \to \Sigma$ is
Bernoulli hence $\alpha$-mixing.

\noindent \textbf{SA3}: We need to check~\cite[equation (4)]
{2020HellaStenlund} that
\[
  \left|\int \varphi (T_{\omega_k} T_{\omega_{k-1}} \cdots T_{\omega_1} x )
    dm - \int \varphi (T_{\omega_k} T_{\omega_{k-1}} \cdots
    T_{\omega_{r+1}} x )dm\right|\le C \eta(k-r)
\]
and 
\[
  \left|\int\varphi \cdot \varphi (T_{\omega_k} T_{\omega_{k-1}} \cdots
    T_{\omega_1} x ) dm - \int \varphi \cdot \varphi (T_{\omega_k}
    T_{\omega_{k-1}} \cdots T_{\omega_{r+1}} x )dm\right|\le C \eta(k-r).
\]


Using the transfer operators, rewrite
\[
  \left|\int \psi \cdot \varphi (T_{\omega_k} T_{\omega_{k-1}} \cdots
    T_{\omega_1} x ) dm - \int \psi \cdot \varphi (T_{\omega_k}
    T_{\omega_{k-1}} \cdots T_{\omega_{r+1}} x )dm\right|
\]
\[
  =\left|\int \varphi \cdot P_{\omega_k} P_{\omega_{k-1}} \cdots
    P_{\omega_1} (\psi) dm - \int \varphi \cdot P_{\omega_k}
    P_{\omega_{k-1}} \cdots P_{\omega_{r+1}} (\psi) dm\right|
\]
\[
  \le \|\varphi\|_{\infty} \| P_{\omega_k} P_{\omega_{k-1}} \cdots
  P_{\omega_{r+1}}[\psi-P_{\omega_{r}} \cdots P_{\omega_1} (\psi)]\|_{L^1}
\]
We have to bound this for $\psi$ either $\mathbf{1}$ or $\varphi$. If
$\psi=\mathbf{1}$ then
\[
  \| P_{\omega_k} P_{\omega_{k-1}} \cdots
  P_{\omega_{r+1}}[\mathbf{1}-P_{\omega_{r}} \cdots P_{\omega_1}
  \mathbf{1}]\|_{L^1}\le C (k-r)^{-\frac{1}{\alpha}+1} (\log
  (k-r))^{\frac{1}{\alpha}}
\]
with $C$ independent of $\omega$ and $r$ by~\cite[Theorem
1.2]{Nicol:2018aa} (see Proposition~\ref{nicoldecay2}) because $\mathbf{1}$
and $P_{\omega_{r}} \cdots P_{\omega_1} \mathbf{1}$ both lie in the cone
and have the same $m$-mean. If $\psi=\varphi$, using
Lemma~\ref{C^1-into-cone}, can write $\varphi-(\int \varphi d m)\mathbf{1}$
as a difference of two functions in the cone, and then the same decay
estimate holds.

\noindent \textbf{SA4}: $(\sigma, \Sigma, \nu)$ is stationary so SA4 is
automatic.
\end{proof}

\subsection{Centering is generically needed in the CLT}

Now we address the question of the necessity of centering in the quenched
central limit theorem. We show that if
$\int \varphi d\mu_{\beta_i}\not = \int \varphi d\mu_{\beta_j}$ for two
maps $T_{\beta_i}$, $T_{\beta_j}$, where $\mu_{\beta_i}$ is the invariant
measure of $T_{\beta_i}$, then centering is needed: although
\[
  \lim_{n\to \infty} \frac{1}{\sqrt{n}} \sum_{j=1}^n \left[\varphi
    (\cT^j_{\omega}) -m( \varphi (\cT^j_{\omega})) \right] \to^d
  N(0,\sigma^2)
\]
for $\nu$-a.e. $\omega$, it is not the case that
\[
  \lim_{n\to \infty} \frac{1}{\sqrt{n}} \sum_{j=1}^n \varphi
  (\cT^j_{\omega} ) \to^d N(0,\sigma^2)
\]
for $\nu$-a.e. $\omega$.

Our proof has the same outline as that of \cite{Abdelkader:2016aa}, adapted
to our setting of polynomial decay of correlations. First we suppose that
the maps $T_{\beta_i}$ do not preserve the same measure. After reindexing
we can suppose that $T_{\beta_1}$ and $T_{\beta_2}$ have different
invariant measures and that
$\int \varphi d\mu_{\beta_1}\not =\int \varphi d\mu_{\beta_2}$, a condition
satisfied by an open and dense set of observables. Recall that the RDS has
the stationary measure $d\mu=hdm$, $h\ge D_{\alpha}>0$ and we have assumed
$\mu (\varphi)=0$, $\varphi \in \cC^1$.

Here are the steps:
\begin{itemize}
\item construct a product random dynamical system on $X\times X$ and prove
  that it satisfies an annealed CLT for
  $\tilde{\varphi} (x,y)=\varphi (x)-\varphi (y)$ with distribution
  $N(0,\tilde{\sigma}^2)$;
\item observe that almost every uncentered quenched CLT has the same
  variance only if $2\sigma^2 =\tilde{\sigma}^2$, where the original RDS
  with stationary measure $d\mu=hdm$ satisfies an annealed CLT for
  $\varphi$ with distribution $N(0,\sigma^2)$;
\item observe that the conclusions of~\cite[Theorem 9]{Abdelkader:2016aa}
  hold in our setting and $\tilde{\sigma}^2=2\sigma^2$ if and only if
  $\lim_{n\to \infty} \frac{1}{n}\int_{\Sigma}\left( \sum_{k=1}^{n-1}
    \int_X \varphi \circ \cT^k_{\omega} hdm\right)^2 d\nu=0$;
\item use ideas of \cite{Abdelkader:2016aa} to show the limit above is zero
  only if a certain function $G$ on $\Sigma$ is a H\"older coboundary,
  which in turn implies
  $\int \varphi d\mu_{\beta_1}= \int \varphi d\mu_{\beta_2}$, a
  contradiction.
\end{itemize}


Let $\varphi\colon X\to\mathbb{R}$ be $\cC^1$, with
$\int_X \varphi d\mu = 0$, and define
$S_{n} (\varphi) = \sum_{k=0}^{n-1} \varphi(\cT_\omega^k x)$ on
$\Sigma\times X$. Recall the standard expression (e.g. see
\cite{Abdelkader:2016aa}) for the annealed variance,
\begin{align*}
  \sigma^{2}=\lim _{n \to \infty} \dfrac{1}{n} \int_{\Sigma} \int_{X}
  [S_{n}(\varphi)]^{2} \ d \mu d \nu. 
\end{align*}

We also consider the product random dynamical system
$(\widetilde{\Sigma}:=\Sigma\times X \times
X,\tilde{\nu}:=\nu\otimes\mu\otimes\mu,\tilde{T})$ defined on $X^2$ by
$\tilde{T}_\omega(x,y) = (T_\omega x, T_\omega y)$. For an observable
$\varphi$, define $\tilde{\varphi}\colon X^2\to\mathbb{R}$ by
$\tilde{\varphi}(x,y)=\varphi(x) - \varphi(y)$, and its Birkhoff sums
$S_n(\tilde{\varphi})$. In Theorem \ref{prod_decay} and Corollary
\ref{prod_clt} of the Appendix we show
$\frac{1}{\sqrt{n}}\sum_{j=1}^n \tilde{\varphi}\circ \tilde{T}^j \to^d
N(0,\tilde{\sigma}^2)$ with respect to $\nu \otimes \mu \otimes \mu$ for
some $\tilde{\sigma}^2\ge 0$.


The following lemma from \cite{Aimino:2015aa} is general and does not
depend upon the underlying dynamics. It is a consequence of Levy's
continuity theorem (Theorem 6.5 in \cite{karr1993}).

\begin{lema*}[{\cite[Lemma 7.2]{Aimino:2015aa}}]
  Assume that $\sigma^2 > 0$ and $\tilde{\sigma}^2 > 0$ are such that
\begin{enumerate}
\item $\frac{S_n(\varphi)}{\sqrt{n}}$ converges in distribution to
  $\mathcal{N}(0,\sigma^2)$ under the probability $\nu\otimes\mu$,

\item $\frac{S_n(\tilde{\varphi})}{\sqrt{n}}$ converges in distribution to
  $\mathcal{N}(0,\tilde{\sigma}^2)$ under the probability
  $\nu\otimes\mu\otimes\mu$,

\item $\frac{S_{n,\omega}(\varphi)}{\sqrt{n}}$ converges in distribution to
  $\mathcal{N}(0,\sigma^2)$ under the probability $\mu$, for $\nu$ almost
  every $\omega$.
\end{enumerate}
Then $2\sigma^2 = \tilde{\sigma}^2$.
\end{lema*}

Suppose two of the maps $T_{\beta_1}$ and $T_{\beta_2}$ have different
invariant measures. It is possible to find a $\cC^1$ $\varphi$ such that
$\int \varphi d\mu_{\beta_1}\not =\int \varphi d\mu_{\beta_2}$. In fact,
$\int \varphi d\mu_{\beta_1}\not =\int \varphi d\mu_{\beta_2}$ for a
$\cC^2$ open and dense set of $\varphi$.

\begin{teo}\label{th.centering}
  Let $\varphi\in C^1$ with $\mu(\phi)=0$ and suppose that
  $\int\varphi~d\mu_{\beta_1}\not =\int\varphi~d\mu_{\beta_2}$. Then it is
  not the case that
    \[
      \lim_{n\to \infty} \frac{1}{\sqrt{n}} \sum_{j=1}^n \varphi
      (\cT^j_{\omega} .) \to N(0,\sigma^2)
    \]
    for almost every $\omega \in \Sigma$. Hence, the Birkhoff sums need to
    be centered along each realization.
\end{teo}

\begin{proof}
  We follow the counterexample method of \cite[Section
  4.3]{Abdelkader:2016aa}. We show that in the uncentered case
  $2\sigma^2\not =\tilde{\sigma}^2$. To do this we use \cite[Theorem
  9]{Abdelkader:2016aa} which holds in our setting,
  namely $\tilde{\sigma}^2=2\sigma^2$ if and only if
  \begin{align}\label{cltcond}
    \lim_{n\to \infty} \int_{\Sigma}
    \left(\frac{1}{\sqrt{n}}\sum_{k=1}^{n-1} \int_X \varphi
    P_{\omega_k}\ldots P_{\omega_n}(h) dm\right)^2 d\nu=0
\end{align}
(as in \cite[Section 4.3]{Abdelkader:2016aa} we change the time direction
and replace $(\omega_1,\omega_2,\ldots, \omega_n)$ by
$(\omega_n,\omega_2,\ldots, \omega_1)$; this does not affect integrals with
respect to $\nu$ over finitely many symbols).

Note that the sequence $P_{\omega_1}P_{\omega_{2}}\ldots P_{\omega_n} h$ is
Cauchy in $L^1$, as $\alpha <\frac{1}{2}$ and
\begin{align*}
  \|P_{\omega_1}P_{\omega_{2}}\ldots P_{\omega_n}(h)
  - P_{\omega_1}P_{\omega_{2}}\ldots P_{\omega_n}\ldots
  P_{\omega_{n+k}}(h)\|_1\le 
  Cn^{-\frac{1}{\alpha}+1} (\log n)^{\frac{1}{\alpha}}
\end{align*}
by Proposition~\ref{nicoldecay2}.
Thus $P_{\omega_1}P_{\omega_{2}}\ldots P_{\omega_n} h\to h_{{\omega}}$ in
$L^1$ for some $h_{{\omega}}\in \cone$. This limit defines $h_{{\omega}}$,
in terms of $\bar{\omega}:=(\ldots, \omega_n,\omega_2,\ldots, \omega_1)$,
i.e. $\omega$ reversed in time. We define
$G(\omega):=\int_X \phi h_{{\omega}}dm$. Note also that
$\|P_{\omega_1}P_{\omega_{2}}\ldots P_{\omega_n} h- h_{{\omega}}\|_1 \le C
n^{-1-\delta}$ for some $\delta >0$, uniformly for ${\omega}\in \Sigma$.
Hence
\begin{align*} 
  \int_\Sigma
  \left(\sum_{k=1}^{n-1}\frac{1}{\sqrt{n}}\int_X \varphi P_{\omega_k}\ldots
  P_{\omega_n} hdm\right)^2 
  & d\nu 
  \\
  = \int_{\Sigma}
  &
    \left(\sum_{k=1}^{n-1}\frac{1}{\sqrt{n}} \left(\int_X \varphi
    h_{{\tau^k\omega}} \; dm +O\left(\sum_{k=1}^{n-1}
    \frac{1}{(n-k)^{1+\delta}}\right)\right)\right)^2 d\nu 
\end{align*}
which gives, using \eqref{cltcond}, that
\begin{align}\label{coboundary_condition}
  \lim_{n\to \infty}
  \int_{\Sigma}\left(\frac{1}{\sqrt{n}}\left(\sum_{k=1}^{n-1} G(\tau^k
  \omega ) \right) \right)^2d\nu=0.
\end{align}


We put a metric on $\Sigma$ by defining $d(\omega, \omega^{'})
=s(\omega,\omega^{'})^{-1-\frac{\epsilon}{2}}$ where
$s(\omega,\omega^{'})=\inf \{n:\omega_n \not = \omega_n^{'}
\}$. With this metric
$\Sigma$ is a compact and complete metric space. Note that
$\|h_{\omega}-h_{\omega^{'}}\|_{L^1} \le C
s(\omega,\omega^{'})^{-\frac{\epsilon}{2}}$ hence
$G({\omega})$ is H\"older with respect to our metric.

As in the Abdulkader-Aimino counterexample,
\eqref{coboundary_condition} implies that
$G=H-H\circ \tau$ for a H\"older function $H$ on the Bernoulli shift
$(\tau,\Sigma,\nu)$: by \cite[Theorem 1.1]{Liverani} (see
Theorem~\ref{th:liverani} in the Appendix) $G$ is a measurable coboundary,
and therefore a H\"older coboundary, by the standard Liv\v{s}ic regularity
theorem (see for instance~\cite[Section 12.2]{OV2016}). Now consider the
points $\beta_1^*:=(\beta_1, \beta_1, \cdots)$ and
$\beta_2^*:=(\beta_2, \beta_2, \cdots)$ in $\Sigma$; they are fixed points
for $\tau$, and correspond to choosing only the map $T_{\beta_1}$,
respectively only the map $T_{\beta_2}$. This implies
$G(\beta_1^*)=G(\beta_2^*)=0$ which in turn implies
$\int \varphi d\mu_{\beta_1} =\int \varphi d\mu_{\beta_2}$, a
contradiction.
\end{proof}

\section{Appendix}

We will show that the system
$\widetilde{F}(\omega,x,y)=(\tau\omega,T_{\omega_1} x,T_{\omega_1} y)$ with
respect to the measure $\nu\otimes\mu^2$ on $\Sigma\times [0,1]^2$ (recall
that $\nu := \bP^{\otimes\bN}$ and $\mu$ is a stationary measure of the
RDS) has summable decay of correlations in $L^2$ for $\alpha<\frac{1}{2}$,
and as a corollary it satisfies the CLT.

\begin{teo}\label{prod_decay}
  Suppose that for $\omega\in \Sigma$, $h =\frac{d \mu}{d m}\in \cone$ and
  each $\phi\in \cC^1$ with $m(\phi h)=0$
  \[
    \|P_{\omega_n}\ldots P_{\omega_1} (\phi h)\|_{L^1(m)} \le C \rho
    (n)(\|\phi\|_{\cC^1} + m(h))
  \]
  (that is, the setting of Proposition~\ref{nicoldecay2}).

  Then there is a constant $\widetilde{C}$, independent of $\omega$, such
  for each $\psi\in \cC^1(X\times X)$ and
  $\varphi \in L^{\infty}(X\times X)$ with $(\mu\otimes\mu)(\psi) = 0$, one
  has
    \begin{align*}
      \left| \int \varphi(\cT_\omega^n x,\cT_\omega^n
      x)\psi(x,y)d\mu(x)d\mu(y)\right|\le
      \widetilde{C} \rho (n) \|\varphi\|_{L^\infty} (\|\psi \|_{\cC^1} +1) 
    \end{align*}
\end{teo}

\begin{proof}
  Since $X\times X$ is compact, $\psi$ is uniformly $\cC^1$ in both
  variables in the sense that $\psi(x_0,y)$ is uniformly $\cC^1$ for each
  $x_0$ and similarly for $\psi(x,y_0)$. We want to estimate
  \begin{align*}
    I := \int \varphi(\cT_\omega^n x,\cT_\omega^n
    y)\psi(x,y)d\mu(x)d\mu(y).
\end{align*}

Define
\begin{align*}
  \ol{\psi}(x) := \int\psi(x,y)d\mu(y), \qquad h_x(y) := \psi(x,y) -
  \ol\psi(x).
\end{align*}
Then $\ol{\psi}, h_x \in \cC^1(X)$, with $\cC^1$-norms bounded by
$2 \|\psi \|_{\cC^1}$, uniformly with respect to $x$.

We can write $I$ as
\begin{align*}
  I = \underbrace{\int \varphi(\cT_\omega^n x,\cT_\omega^n
  y)\left[\psi(x,y) -\ol\psi(x,y)
  \right]d\mu(x)d\mu(y)}_{:=I_1} \\  
  + \underbrace{\int \varphi(\cT_\omega^n x,\cT_\omega^n
  y)\ol\psi(x,y)d\mu(x)d\mu(y)}_{:=I_2}.
\end{align*}

Define now $g_{\omega,x}(y) := \varphi(\cT_\omega^n x, y)$. Then (note that
$\int h_x(y)h(y) d m(y)=0$)
\begin{align*}
  |I_1| = \left|\int \left(\int  g_{\omega,x}(\cT_\omega^n
  y) h_x(y)h(y) d m(y)\right) d\mu(x)\right|
  &= \left|\int \left(\int g_{\omega,x}(y)\cP_\omega^n(h_x(y)h(y))
    d m(y)\right) d \mu(x)\right|\\
  & \leq \|\varphi \|_{L^\infty} \sup_{x}\|\cP_\omega^n(h_x(y)h(y))\|_{L^1(m(y))} \\
  & \leq C' \|\varphi \|_{L^\infty} (\|\psi\|_{\cC^1} + m(h))\rho (n).
\end{align*}
by the hypothesis.

Similarly, define $k_{\omega,y}(x) := \varphi(x,\cT_\omega^n y)$ so then
(again, $\int \ol\psi(x)h(x) d m(x) = 0$)
\begin{align*}
  |I_2| &=  \left|\int \left(\int k_{\omega,y}(\cT_\omega^n
          x)\ol\psi(x)d\mu(x)\right) d\mu(y) \right| \\
        & = \left|\int \left( \int k_{\omega,y}(x)
          \cP_\omega^n(\ol\psi(x)h(x))dm(x) \right) d\mu(y)\right|\\
        &\leq \|\varphi \|_{L^\infty} \|\cP_\omega^n(\ol\psi(x)h(x))\|_{L^1(m(x))} \\
        & \leq C' \|\varphi \|_{L^\infty} (\|\psi\|_{\cC^1} + m(h)) \rho (n).
\end{align*}

These imply that
$| I | \leq 2C' \|\varphi\|_{L^\infty} (\|\psi\|_{\cC^1} + m(h)) \rho (n)$.
\end{proof}

\begin{coro}\label{prod_clt}
  Under the assumptions of Theorem~\ref{prod_decay}, 
  for $\psi \in \cC^1(X\times X)$ with $(\mu \otimes \mu)(\psi)=0$,
  $\frac{1}{\sqrt n} \sum_{k=1}^n \psi \circ {\widetilde{F}}^k(\omega, x,
  y)$ satisfies a CLT with respect to $\nu \otimes \mu \otimes \mu$, that
  is
  \[
    \frac{1}{\sqrt n} \sum_{k=1}^n \psi \circ {\widetilde{F}}^k(\omega, x,
    y) \to^d N(0,\widetilde{\sigma}^2)
  \]
  in distribution for some $\widetilde{\sigma}^2\ge 0$.
\end{coro}

\begin{proof}
  Let $Q$ be the adjoint of
  $\tilde{F}(\omega,x,y)=(\sigma\omega,T_{\omega_1} x,T_{\omega_1} y)$ with
  respect to the invariant measure $\nu\otimes\mu\otimes\mu$ on
  $\Sigma\times X^2$ so that
  \begin{align*}
    \int \varphi \circ \tilde{F}(\omega, x, y)\psi(\omega, x,y)d\mu(x)d\mu(y)d\nu(\omega)
    = \int \varphi(\omega, x,y)(Q\psi)(\omega, x,y)d\mu(x)d\mu(y)d\nu(\omega).
  \end{align*}
  for $\phi\in L^{\infty}(\Sigma\times X \times X)$. Iterating we have
  \begin{align*}
    \int \varphi \circ \tilde{F}^n( \omega, x, y)\psi(\omega,x,y) d\mu(x)d\mu(y)d\nu(\omega)
    = \int \varphi(\omega, x,y) (Q^n\psi)(\omega, x,y)d\mu(x)d\mu(y)d\nu(\omega).
  \end{align*}
  Taking $\varphi = \sign(Q^n\psi)$, we see from Theorem~\ref{prod_decay}
  that $\|Q^n\psi \|_{L^1} \leq {C'} \rho (n)$.

  The proof now follows, as in Proposition~\ref{annealed_CLT},
  from~\cite[Theorem 1.1]{Liverani} (see Theorem~\ref{th:liverani} in the
  Appendix).
\end{proof}

\begin{proof}[Proof of Lemma~\ref{C^1-into-cone}]

  Let $f_1 = (\varphi+\lambda x+A)h+B$ and $f_2 = (A+\lambda x )h+B$.

  First we show that $f_1\in{\cone}$. It is clear that
  $f_1\in \mathcal{C}^0(0,1]\cap L^1(m)$. Choose $\lambda < 0$ such that
  $|\lambda| > \|\varphi'\|_{L^\infty}$ and $A>0$ large enough so that
\begin{align*}
  \varphi + \lambda x + A > 0.
\end{align*}

This ensures that $f_1\geq 0$ for any value of $B\geq 0$. Note now that
\begin{align*}
  (\varphi + \lambda x + A)' = \varphi'+ \lambda \leq 0
\end{align*}
so $\varphi + \lambda x + A$ is decreasing. Since both
$\varphi + \lambda x + A$ and $h$ are positive and decreasing, we obtain
that $f_1$ is decreasing as well. We show now that $x^{\alpha + 1}f_2$ is
increasing. Since $h\in{\cone}$, $h$ is non-increasing so $h'$ exists
$m$-a.e. and $h'\leq 0$ $m$-a.e. Then $(x^{\alpha +1} h)'$ exists $m$-a.e.
as well, and we can compute this derivative as
\begin{align*}
  (x^{\alpha+1} h)' = (\alpha + 1)x^\alpha h + x^{\alpha + 1}h' \geq 0
\end{align*}
because it is increasing.

We compute now the derivative of $x^{\alpha + 1}f_2$:
\begin{align*}
  (x^{\alpha+1}[(\varphi+\lambda x+A)h+B])' = (\alpha + 1)x^\alpha \varphi
  h + x^{\alpha+1}\varphi' h + x^{\alpha+1}\varphi h' + (\alpha +
  2)x^{\alpha+1}h\lambda +  \\ 
  \lambda x^{\alpha+2}h' + (\alpha + 1)A x^\alpha h + A x^{\alpha+1}h' +
  (\alpha+1)x^\alpha B. 
\end{align*}
We group terms conveniently: note that
\begin{align*}
  (\alpha + 1)x^\alpha \varphi h + (\alpha + 1)A x^\alpha h +
  x^{\alpha+1}\varphi h' + A x^{\alpha + 1}h' = (\varphi + A)[(\alpha +
  1)x^\alpha h + h'x^{\alpha+1}]\geq 0
\end{align*}
$m$-a.e., since the term in the square brackets corresponds to
$(x^{\alpha+1}h)'\geq 0$. The term $\lambda x^{\alpha+2}h'$ is non-negative
$m$-a.e. since $\lambda, h'\leq 0$. Since
$0\leq h(x) x^\alpha \leq a m(h)$, we have
$0\leq -x^{\alpha+1}h' \leq (\alpha+1)x^\alpha h \leq (\alpha + 1)am(h)$
and then the terms
$(\alpha+2)\lambda x^{\alpha+1}h + x^{\alpha+1}h\varphi'$ are bounded.
Thus, we can take $B > 0$ big enough so that
\begin{align*}
  (\alpha+1)x^{\alpha}B \ge (\alpha+2)\lambda x^{\alpha+1}h  +
  x^{\alpha+1}h\varphi'.
\end{align*}
With this, we have that $(x^{\alpha +1} h)'\geq 0$ and so $x^{\alpha +1} h$
is increasing.

Finally, we check that $f_1(x) x^{\alpha} \le a m(f_1)$. Using that
$h(x) x^\alpha \le a m(h)$,
\begin{eqnarray*}
  [(\phi+\lambda x + A) h + B] x^\alpha \le (\phi+\lambda x + A) h
  x^\alpha + B \le \sup (\phi+\lambda x + A)a m(h) + B.
\end{eqnarray*}
On the other hand,
$a m((\phi+\lambda x + A) h + B)\ge a \inf(\phi+\lambda x + A) m(h) + a B$,
so it suffices to have
\begin{eqnarray*} \sup (\phi+\lambda x + A)a m(h) + B \le a
  \inf(\phi+\lambda x + A) m(h) + a B \\ \iff B \ge \frac{a}{a-1} \big[
  \sup (\phi+\lambda x + A) - \inf(\phi+\lambda x + A) \big]m(h).
\end{eqnarray*}

Thus, we see that $f_1 \in{\cone}$. The proof that $f_2\in{\cone}$ is the
same, take $\phi(x)\equiv 0$.
\end{proof}

\begin{teo}[{special case of \cite[Theorem
    1.1]{Liverani}}]\label{th:liverani}
  Assume $T:Y \to Y$ preserves the probability measure $\eta$ on the
  $\sigma$-algebra $\cB$. Denote by $P$ its transfer operator.

  If $\phi\in L^\infty({\eta})$ with $\eta(\phi)=0$ and
  $\sum_k \|P^k \phi\|_{L^1(\eta)} < \infty$ then a central limit theorem
  holds for $S_{n}\varphi:=\sum_{k=1}^n \phi\circ T^k$ with respect to the
  measure $\eta$, that is, $\frac{1}{\sqrt{n}} S_{n}\varphi$ converges in
  distribution to $\mathcal{N}(0,\sigma^2)$. The variance is given by
  \[
    \sigma^2=-\eta(\varphi^2) +2 \sum_{k=0}^{\infty}\eta(\varphi \cdot
    \varphi \circ T^k).
  \]
  In addition, $\sigma^2=0$ iff $\phi\circ T$ is a measurable coboundary,
  that is $\phi\circ T=g - g\circ T$ for a measurable~$g$.
\end{teo}

\paragraph{Acknowledgements}

We thank Ian Melbourne and the anonymous referee for many helpful comments and suggestions. MN was supported in part by NSF Grant DMS 1600780. FPP thanks the University of Houston for hospitality while this work was completed. FPP
was partially supported by the Becas Chile scholarship scheme from CONICYT. AT
was supported in part by NSF Grant DMS 1816315.

\bibliographystyle{alpha} \bibliography{biblio}
\end{document}